\documentclass{amsart} 

\newtheorem{theorem}{Theorem}[section]
\newtheorem{lemma}[theorem]{Lemma}
\newtheorem{proposition}[theorem]{Proposition}

\theoremstyle{definition}
\newtheorem{definition}[theorem]{Definition}
\newtheorem{example}[theorem]{Example}
\newtheorem{remark}[theorem]{Remark}

\newcommand{\Z}{\mathbb{Z}}
\newcommand{\R}{\mathbb{R}}

\newcommand{\III}{I\!I\!I}

\usepackage{longtable,caption} 
\usepackage{graphicx}

\begin{document}


\title[The intersection polynomials of a virtual knot]
{The intersection polynomials of a virtual knot I:
Definitions and calculations}

\author{Ryuji HIGA}
\address{Department of Mathematics, Kobe University, 
Rokkodai-cho 1-1, Nada-ku, Kobe 657-8501, Japan}
\email{higa@math.kobe-u.ac.jp}

\author{Takuji NAKAMURA}
\address{Faculty of Education, 
University of Yamanashi,
Takeda 4-4-37, Kofu, Yamanashi, 400-8510, Japan}
\email{takunakamura@yamanashi.ac.jp}

\author{Yasutaka NAKANISHI}
\address{Department of Mathematics, Kobe University, 
Rokkodai-cho 1-1, Nada-ku, Kobe 657-8501, Japan}
\email{nakanisi@math.kobe-u.ac.jp}

\author{Shin SATOH}
\address{Department of Mathematics, Kobe University, 
Rokkodai-cho 1-1, Nada-ku, Kobe 657-8501, Japan}
\email{shin@math.kobe-u.ac.jp}

\renewcommand{\thefootnote}{\fnsymbol{footnote}}
\footnote[0]{This work was supported by 
JSPS KAKENHI Grant Numbers 
JP20K03621, JP19K03492, and 
JP19K03466.}


\renewcommand{\thefootnote}{\fnsymbol{footnote}}




\maketitle


\begin{abstract} 
We introduce three kinds of invariants of a virtual knot 
called the first, second, and third intersection polynomials. 
The definition is based on the intersection number of 
a pair of curves on a closed surface. 
The calculations of intersection polynomials
are given up to crossing number four. 
We also study several properties of intersection polynomials. 
\end{abstract}


\section{Introduction}\label{sec1} 

In classical knot theory, 
we study a circle embedded in a $3$-dimensional sphere $S^3$ 
under an ambient isotopy. 
It is enough to consider the product 
$S^2\times I$ of a $2$-sphere $S^2$ and an interval $I$ 
instead of $S^3$. 
In this sense, it is natural to study a circle in the product 
$\Sigma_g\times I$ of a closed, connected, oriented surface 
$\Sigma_g$ of genus $g$ and $I$. 
Kauffman \cite{K1} 
leads the unification of such knot theories 
for all genera and introduces virtual knot theory. 
A virtual knot is described by a diagram on $\Sigma_g$ 
for some $g\geq 0$ under the projection 
$\Sigma_g\times I$ onto $\Sigma_g$. 
We are allowed to use three kinds of 
Reidemeister moves for diagrams and (de)stabilizations for surfaces. 
See Figure~\ref{stab2}.

\begin{figure}[htb]
\begin{center}
\includegraphics[bb=0 0 261.05 71.93]{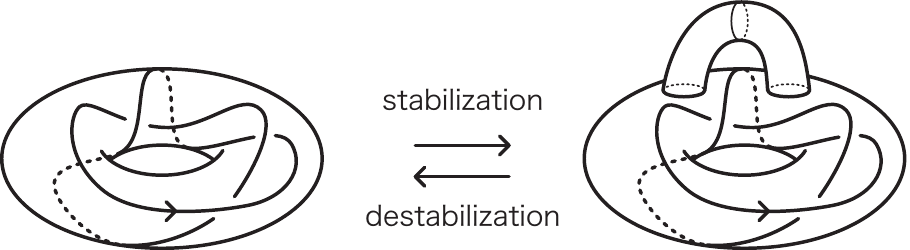}
\caption{}
\label{stab2}
\end{center}
\end{figure}

Some invariants of a virtual knot are natural generalizations of 
those of a classical knot such as knot groups and Jones polynomials \cite{K1}, 
and some vanish for classical knots 
such as  Sawollek polynomials \cite{Saw} and writhe polynomials 
\cite{CG,FK,K3,ST}. 
In this paper, 
we will introduce three kinds of new invariants of the latter type. 

This paper is organized as follows. 
In Section~\ref{sec2}, 
we define three kinds of Laurent polynomials 
$I_K(t)$, $I\!I_K(t)$, and ${\III}_K(t)$ of a virtual knot $K$ 
and prove the following. 
\begin{theorem}
$I_K(t)$, $I\!I_K(t)$, 
and ${\III}_K(t)$ are invariants of $K$. 
\end{theorem}
{\noindent}These invariants are called the 
{\it first}, {\it second}, and {\it third intersection polynomials} of $K$, 
respectively. 
We remark that they are trivial for classical knots (Lemma~\ref{lem211}).

In Section~\ref{sec3}, 
we explain how to calculate the intersection polynomials. 
We give the intersection polynomials 
for all the virtual knots up to crossing number four 
in Appendices \ref{appA} and \ref{appB} (Theorem~\ref{thm35}). 
By observing these calculations, 
we have the following. 
\begin{theorem}
$I_K(t)$, $I\!I_K(t)$, 
and ${\III}_K(t)$ are independent of each other. 
\end{theorem}

Sections~\ref{sec4} and \ref{sec5} 
are devoted to giving several applications of the intersection polynomials. 
In Section~\ref{sec4}, 
we study the behaviors of intersection polynomials 
on symmetry of a virtual knot. 
Let $-K$, $K^{\#}$, and 
$K^*$ be the reverse, the vertical mirror image, 
and the horizontal mirror image of $K$, respectively. 
There are known several examples of virtual knots $K$ 
such that the eight knots derived from $K$ are mutually distinct (cf.~\cite{KS}). 
In this paper, we prove the following. 
\begin{theorem}
There are infinitely many virtual knot $K$ such that 
$$K,-K,K^\#,K^*, 
-K^\#, -K^*, K^{\#*}, 
\mbox{ and } -K^{\#*}$$ 
are mutually distinct. 
\end{theorem}

{\noindent}In Section~\ref{sec5}, 
we give lower bounds of the crossing number ${\rm c}(K)$ 
and the virtual crossing number ${\rm vc}(K)$ of a virtual knot $K$ 
by intersection polynomials. 
\begin{theorem}
For any virtual knot $K$, we have
\begin{itemize}
\item[{\rm (i)}] 
${\rm c}(K)\geq {\rm deg}X_K(t)+1$ $(X=I,I\! I,{\III})$ \mbox{ and }
 \vspace{1mm}
\item[{\rm (ii)}] 
${\rm vc}(K)\geq {\rm deg}X_K(t)$ $(X=I,I\! I,{\III})$. 
\end{itemize}
\end{theorem}

In Appendix~\ref{appA}, 
we give the table of $I_K(t)$ and $I\!I_K(t)$ with ${\rm c}(K)\leq 4$. 
In Appendix~\ref{appB}, 
we give the table of ${\III}_K(t)$ with ${\rm c}(K)\leq 4$. 

In forthcoming papers, we will study 
the behavior of intersection polynomials under 
a connected sum \cite{HNNS2}, 
a characterization of intersection polynomials \cite{HNNS3}, and 
a relationship with crossing changes \cite{HNNS4}.


\section{Definitions}\label{sec2}

Let $\Sigma_g$ be a closed, connected, oriented surface 
of genus $g$, 
and $\alpha$ and $\beta$ closed, oriented curves on 
$\Sigma_g$.
We often regard these curves as homology cycles of $H_1(\Sigma_g)$. 
The {\it intersection number} $\alpha\cdot \beta\in{\Z}$ 
is defined to be the homology intersection 
of the ordered pair $(\alpha,\beta)$. 
Geometrically it is calculated as follows. 
By perturbing $\alpha$ and $\beta$ if necessary,  
we may assume that $\alpha\cap \beta$ consists of 
$m$ transverse double points $p_1,\dots,p_m$. 
At a double point $p_k$ $(1\leq k\leq m)$, 
if $\beta$ intersects $\alpha$ from the left or right 
as we walk along $\alpha$, 
we define $e_k=+1$ or $-1$, respectively. 
Then we have $\alpha\cdot \beta=\sum_{k=1}^m e_k$. 
See Figure~\ref{H2}. 
We remark that $\alpha\cdot\beta=-\beta\cdot\alpha$ 
and $\alpha\cdot\alpha=0$ by definition.

\begin{figure}[htb]
\begin{center}
\includegraphics[bb = 0 0 296.43 55.84]{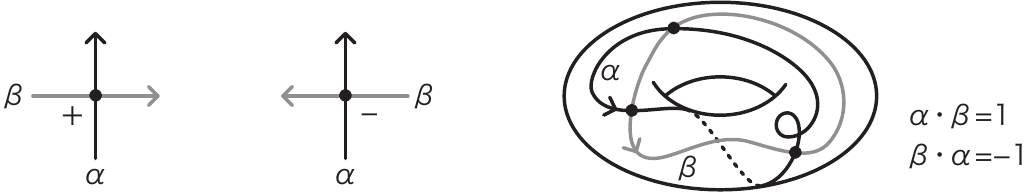}
\caption{}
\label{H2}
\end{center}
\end{figure}

We consider a circle embedded in $\Sigma_g\times[0,1]$ for some $g\geq 0$. 
We identify two embedded circles up to 
ambient isotopies and (de)stabilizations. 
Such an equivalence class is called a {\it virtual knot} 
(cf. \cite{CKS,KK,K1,Kurp}). 

More precisely, 
a virtual knot is described by a {\it diagram} on $\Sigma_g$ 
which is a projection image 
under the projection $\Sigma_g\times[0,1]\rightarrow\Sigma_g$ 
equipped with over/under-information at double points. 
A double point with over/under information 
is called a {\it crossing}. 
Two diagrams $(\Sigma_g,D)$ and 
$(\Sigma_{g'},D')$ present the same virtual knot 
if and only if there is a finite sequence of diagrams  
$$(\Sigma_g,D)=(\Sigma_{g_1},D_1),(\Sigma_{g_2},D_2),\dots,
(\Sigma_{g_s},D_s)=(\Sigma_{g'},D')$$
such that for each $1\leq i\leq s-1$, 
\begin{itemize} 
\item[(0)] 
$g_{i+1}=g_i$ holds and 
$(\Sigma_{g_{i+1}},D_{i+1})$ is obtained from 
$(\Sigma_{g_i},D_i)$ by 
an orientation-preserving homeomorphism 
of $\Sigma_{g_i}=\Sigma_{g_{i+1}}$, 
\item[(i)] 
$g_{i+1}=g_i$ holds and 
$D_{i+1}$ is obtained from $D_i$ by 
a Reidemeister move on $\Sigma_{g_i}=\Sigma_{g_{i+1}}$, or 
\item[(ii)] 
$g_{i+1}=g_i\pm 1$ holds and 
$\Sigma_{g_{i+1}}$ is obtained from $\Sigma_{g_i}$ 
by $1$- or $2$-handle surgery 
missing $D_i=D_{i+1}$. 
Such a deformation is called a {\it stabilization} or 
{\it destabilization}, respectively. 
\end{itemize}
Throughout this paper, 
we assume that all virtual knots are oriented.

Let $D=(\Sigma,D)$ be a diagram of a virtual knot $K$, 
and $c_1,\dots,c_n$ the crossings of $D$. 
We denote by $\gamma_D$ the closed, oriented curve on $\Sigma$ 
obtained from $D$ by ignoring over/under-information at $c_i$'s. 
Furthermore, we denote by $\gamma_i$ $(1\leq i\leq n)$ 
the closed, oriented curve as a part of $\gamma_D$ 
from the overcrossing to the undercrossing at $c_i$, 
and by $\overline{\gamma}_i$ 
the curve from the undercrossing to the overcrossing at $c_i$. 
See Figure~\ref{gamma}. 
These curves satisfy 
$\gamma_i+\overline{\gamma}_i=\gamma_D$ 
and $\gamma_i\cdot\overline{\gamma}_i=
\gamma_i\cdot(\gamma_D-\gamma_i)
=\gamma_i\cdot\gamma_D$ 
as homology cycles on $\Sigma$. 
We call $\gamma_i$ and $\overline{\gamma}_i$ 
the {\it cycles} at $c_i$ on $\Sigma$.

\begin{figure}[htb]
\begin{center}
\includegraphics[bb = 0 0 337.93 34.23]{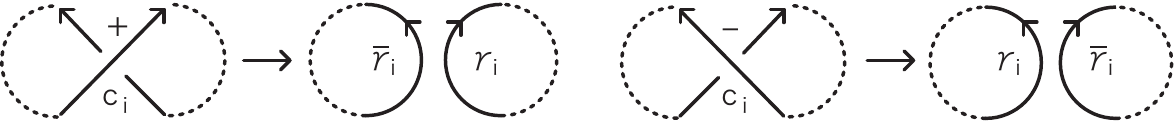}
\caption{}
\label{gamma}
\end{center}
\end{figure}

\begin{definition}[\cite{CG,FK,K3,ST}]\label{def21}
The Laurent polynomial 
$$W_D(t)
=\sum_{i=1}^n \varepsilon_i(t^{\gamma_i\cdot\overline{\gamma}_i}-1)
=\sum_{i=1}^n \varepsilon_i(t^{\gamma_i\cdot\gamma_D}-1)
\in{\Z}[t,t^{-1}]$$
is an invariant of $K$, 
where $\varepsilon_i$ is the sign of $c_i$. 
It is the {\it writhe polynomial} of $K$ and denoted by $W_K(t)$. 
The exponent $\gamma_i\cdot\overline{\gamma}_i$ 
is called the {\it index} of a crossing $c_i$. 
\end{definition}

Let $D^\#$ be the diagram obtained from $D$ by changing 
over/under-information at every crossing of $D$. 

\begin{lemma}[\cite{ST}]\label{lem22}
For any diagram $D$ on $\Sigma$, 
we have 
$W_{D^\#}(t)=-W_D(t^{-1})$. 
\end{lemma}

\begin{proof} 
Let $c_1^\#,\dots,c_n^\#$ be the crossings of $D^\#$ 
such that $c_i^\#$ corresponds to $c_i$, 
$\varepsilon_i^\#$ the sign of $c_i^\#$, 
and $\gamma_i^\#$ the cycle at $c_i^\#$ on $\Sigma$ $(1\leq i\leq n)$. 
Then we have 
$$\varepsilon_i^\#=-\varepsilon_i, \ 
\gamma_i^\#=\overline{\gamma}_i, \mbox{ and }
\overline{\gamma}_i^\#=\gamma_i.$$
Therefore it holds that 
$$W_{D^\#}(t)=\sum_{i=1}^n 
\varepsilon_i^\#(t^{\gamma_i^\#\cdot\overline{\gamma}_i^\#}-1)
=-\sum_{i=1}^n\varepsilon_i(t^{\overline{\gamma}_i\cdot\gamma_i}-1)
=-W_D(t^{-1}).$$
\end{proof}

Now we consider four kinds of Laurent polynomials as follows; 
$$\begin{array}{l} 
\displaystyle{
f_{01}(D;t)=\sum_{1\leq i,j\leq n} \varepsilon_i\varepsilon_j(t^{\gamma_i\cdot\overline{\gamma}_j}-1)}, 
\quad
\displaystyle{
f_{10}(D;t)=\sum_{1\leq i,j\leq n} \varepsilon_i\varepsilon_j(t^{\overline{\gamma}_i\cdot\gamma_j}-1)}, \\
\displaystyle{
f_{00}(D;t)=\sum_{1\leq i,j\leq n} \varepsilon_i\varepsilon_j(t^{\gamma_i\cdot\gamma_j}-1)}, 
\quad
f_{11}(D;t)=\sum_{1\leq i,j\leq n} \varepsilon_i\varepsilon_j(t^{\overline{\gamma}_i\cdot
\overline{\gamma}_j}-1). 
\end{array}$$

\begin{lemma}\label{lem23}
For any diagram $D$ on $\Sigma$, we have the following. 
\begin{itemize}
\item[{\rm (i)}] 
$f_{10}(D;t)=f_{01}(D^\#;t)$. 
\item[{\rm (ii)}] 
$f_{11}(D;t)=f_{00}(D^\#;t)$. 
\end{itemize}
\end{lemma}

\begin{proof}
Let $c_1^\#,\dots,c_n^\#$ be the crossings of $D^\#$ 
such that $c_i^\#$ corresponds to $c_i$, 
$\varepsilon_i^\#$ the sign of $c_i^\#$, 
and $\gamma_i^\#$ the cycle at $c_i^\#$ on $\Sigma$ $(1\leq i\leq n)$. 
Then we have 
$$\varepsilon_i^\#=-\varepsilon_i, \ 
\gamma_i^\#=\overline{\gamma}_i, \mbox{ and }
\overline{\gamma}_i^\#=\gamma_i.$$
Since it holds that 
$$\varepsilon_i\varepsilon_j=
\varepsilon_i^\#\varepsilon_j^\#, \ 
\overline{\gamma}_i\cdot\gamma_j=
\gamma_i^\#\cdot\overline{\gamma}_j^\#, \mbox{ and }
\overline{\gamma}_i\cdot\overline{\gamma}_j
=\gamma_i^\#\cdot\gamma_j^\#,$$ 
we have the conclusion by definition. 
\end{proof}

In what follows, 
we often abbreviate $f_{pq}(D;t)$ to $f_{pq}(D)$ 
for $p,q\in\{0,1\}$.

\begin{lemma}\label{lem24}
If a diagram $D'$ is obtained from $D$ 
by a second or third Reidemeister move on $\Sigma$, 
then we have 
$f_{pq}(D)=f_{pq}(D')$ for any $p,q\in\{0,1\}$. 
\end{lemma}

\begin{proof}
Since $D'^\#$ is obtained from $D^\#$ 
by a second or third Reidemeister move on $\Sigma$, 
it is sufficient to prove the invariance of $f_{01}(D)$ and $f_{00}(D)$ 
by Lemma~\ref{lem23}. 

\underline{A second Reidemeister move.} 
Assume that $D'$ is obtained from $D$ 
by a second Reidemeister move removing a pair of 
crossings $c_1$ and $c_2$ of $D$. 
For $3\leq i\leq n$, 
let $c_i'$ be the crossing of $D'$ corresponding to $c_i$, 
$\varepsilon_i'$ the sign of $c_i'$, and 
$\gamma_i'$ the cycle at $c_i'$ on $\Sigma$. 
Then it holds that 
$$\varepsilon_1=-\varepsilon_2, \ 
\gamma_1=\gamma_2, \mbox{ and }
\varepsilon_i'=\varepsilon_i, \ 
\gamma_i'=\gamma_i \ (3\leq i\leq n).$$
See Figure~\ref{R2}. 
A disk on $\Sigma$ is bounded by 
two arcs on $D$ connecting $c_1$ and $c_2$. 
Since $\gamma_{D'}=\gamma_D$ and 
$\overline{\gamma}_i'=\overline{\gamma}_i$ $(3\leq i\leq n)$, 
we have 
\begin{eqnarray*}
\lefteqn{f_{01}(D)-f_{01}(D')}\\
&=&
\varepsilon_1^2(t^{\gamma_1\cdot\overline{\gamma}_1}-1)
+\varepsilon_1\varepsilon_2(t^{\gamma_1\cdot
\overline{\gamma}_2}-1)
+\varepsilon_2\varepsilon_1(t^{\gamma_2\cdot
\overline{\gamma}_1}-1)
+\varepsilon_2^2(t^{\gamma_2\cdot\overline{\gamma}_2}-1)\\
&&
+\sum_{j=3}^n
\varepsilon_1\varepsilon_{j}(t^{\gamma_1\cdot
\overline{\gamma}_j}-1)
+\sum_{j=3}^n
\varepsilon_2\varepsilon_{j}(t^{\gamma_2\cdot
\overline{\gamma}_j}-1)\\
&&
+\sum_{i=3}^n
\varepsilon_{i}\varepsilon_1(t^{\gamma_{i}\cdot
\overline{\gamma}_1}-1)
+\sum_{i=3}^n
\varepsilon_{i}\varepsilon_2(t^{\gamma_{i}\cdot
\overline{\gamma}_2}-1)\\
&=&
(t^{\gamma_1\cdot\overline{\gamma}_1}-1)
-(t^{\gamma_1\cdot\overline{\gamma}_1}-1)
-(t^{\gamma_1\cdot\overline{\gamma}_1}-1)
+(t^{\gamma_1\cdot\overline{\gamma}_1}-1)=0. 
\end{eqnarray*}
Therefore $f_{01}(D)$ is invariant under a second Reidemeister move. 
On the other hand, 
the invariance of $f_{00}(D)$ is proved by 
$$f_{00}(D)-f_{00}(D')=
(t^{\gamma_1\cdot\gamma_1}-1) 
-(t^{\gamma_1\cdot\gamma_1}-1)
-(t^{\gamma_1\cdot\gamma_1}-1)
+(t^{\gamma_1\cdot\gamma_1}-1)=0. $$

\begin{figure}[htb]
\begin{center}
\includegraphics[bb = 0 0 171.11 77.69]{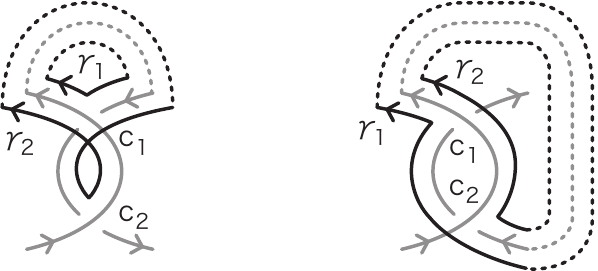}
\caption{}
\label{R2}
\end{center}
\end{figure}

\underline{A third Reidemeister move.} 
Assume that $D'$ is obtained from $D$ 
by a third Reidemeister move involving three crossings 
$c_1,c_2$, and $c_3$ of $D$. 
For $1\leq i\leq n$, 
let $c_i'$ be the crossing corresponding to $c_i$, 
$\varepsilon_i'$ the sign of $c_i'$, and 
$\gamma_i'$ the cycle at $c_i'$ on $\Sigma$. 
Then it holds that 
$$\varepsilon_i'=\varepsilon_i \mbox{ and }
\gamma_i'=\gamma_i \ (1\leq i\leq n).$$
Figure~\ref{R3} shows $\gamma_i'=\gamma_i$ for $i=1,2,3$. 
A disk on $\Sigma$ is bounded by 
three arcs of $D$ connecting $c_1$ and $c_2$, 
$c_1$ and $c_3$, and $c_2$ and $c_3$. 
Therefore both $f_{01}(D)$ and $f_{00}(D)$ are invariant 
under a third Reidemeister move. 
We remark that $\gamma_1+\gamma_2=\gamma_3$ holds 
in this case. 
\end{proof}

\begin{figure}[htb]
\begin{center}
\includegraphics[bb = 0 0 317.96 93.4]{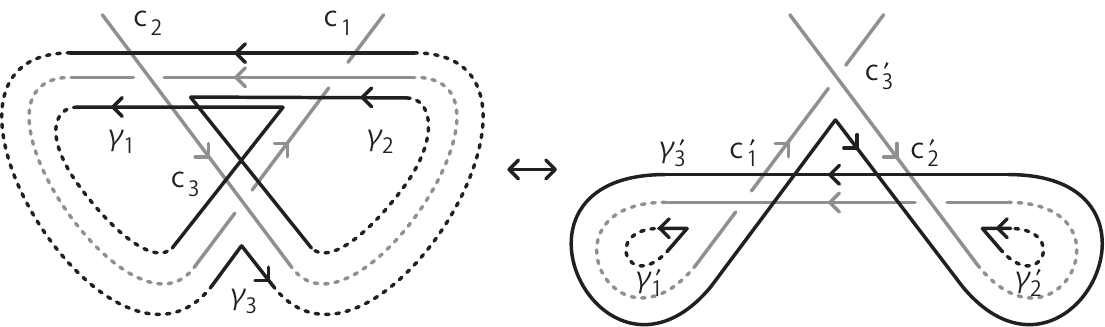}
\caption{}
\label{R3}
\end{center}
\end{figure}

We use the notation 
$\overline{W}_K(t)=\overline{W}_D(t)=W_D(t)+W_D(t^{-1})$. 
It follows by Lemma~\ref{lem22} that 
$$\overline{W}_{D^\#}(t)=
W_{D^\#}(t)+W_{D^\#}(t^{-1})
=-W_D(t^{-1})-W_D(t)=-\overline{W}_D(t).$$

\begin{lemma}\label{lem25}
If a diagram $D'$ is obtained from $D$ 
by a first Reidemeister move on $\Sigma$ 
as shown in {\rm Figure~\ref{R1}(a)--(d)}, 
then the difference $f_{pq}(D)-f_{pq}(D')$ 
is given as shown in {\rm Table~\ref{table1}}. 
\end{lemma}

\begin{figure}[htb]
\begin{center}
\includegraphics[bb = 0 0 300.48 62.09]{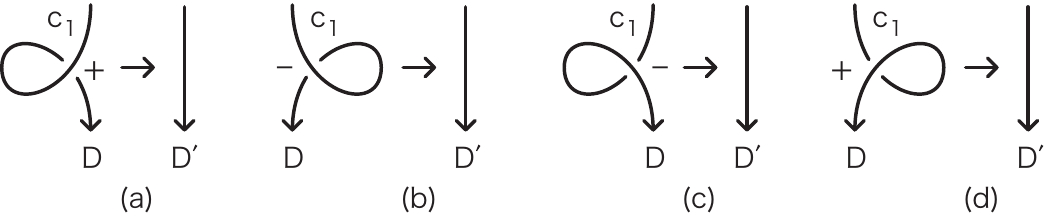}
\caption{}
\label{R1}
\end{center}
\end{figure}

\begin{table}[thb]
\begin{center}
\renewcommand{\arraystretch}{1.3}
\begin{tabular}{c|c|c|c|c}
 & (a) $\varepsilon_1=+1$ & 
 (b) $\varepsilon_1=-1$ & 
 (c) $\varepsilon_1=-1$ & 
 (d) $\varepsilon_1=+1$ \\
\hline\hline
$f_{01}(D)-f_{01}(D')$ & $W_K(t)$ & $-W_K(t)$ & 
$-W_K(t)$ & $W_K(t)$ \\
\hline
$f_{10}(D)-f_{10}(D')$ & $W_K(t^{-1})$ & $-W_K(t^{-1})$ & 
$-W_K(t^{-1})$ & $W_K(t^{-1})$ \\
\hline
$f_{00}(D)-f_{00}(D')$ & $0$ & $0$ & 
$-\overline{W}_K(t)$ & $\overline{W}_K(t)$ \\
\hline
$f_{11}(D)-f_{11}(D')$ & $\overline{W}_K(t)$ & $-\overline{W}_K(t)$ & 
$0$ & $0$
\end{tabular}
\end{center}
\caption{}
\label{table1}
\end{table}

\begin{proof}
Assume that $c_1$ is removed from $D$ by a first Reidemeister move. 
For $2\leq i\leq n$, 
let $c_i'$ be the crossing of $D'$ corresponding to $c_i$, 
$\varepsilon_i'$ the sign of $c_i'$, and 
$\gamma_i'$ the cycle at $c_i'$ on $\Sigma$. 
Then it holds that 
$\gamma_i'=\gamma_i$ and 
$\varepsilon_i'=\varepsilon_i$ $(2\leq i\leq n)$. 
By definition, we have 
$\gamma_1=0$ and $\overline{\gamma}_1=\gamma_{D}$ 
for (a) and (b), and 
$\gamma_1=\gamma_D$ and $\overline{\gamma}_1=0$ 
for (c) and (d).

\underline{$(p,q)=(0,1)$.} 
It holds that 
$$f_{01}(D)-f_{01}(D')=
\varepsilon_1^2(t^{\gamma_1\cdot\overline{\gamma}_1}-1)
+\sum_{j=2}^n
\varepsilon_1\varepsilon_{j}(t^{\gamma_1\cdot
\overline{\gamma}_j}-1)
+\sum_{i=2}^n
\varepsilon_{i}\varepsilon_1(t^{\gamma_{i}\cdot
\overline{\gamma}_1}-1).$$
For (a) and (b), 
we have 
$$f_{01}(D)-f_{01}(D')=
\varepsilon_1\sum_{i=2}^n
\varepsilon_{i}(t^{\gamma_{i}\cdot
\gamma_{D}}-1)
=
\varepsilon_1\sum_{i=2}^n
\varepsilon_{i}(t^{\gamma_{i}\cdot
\overline{\gamma}_i}-1)=\varepsilon_1W_K(t).$$
For (c) and (d), by using the equation 
$\gamma_{D}\cdot\overline{\gamma}_j
=(\gamma_j+\overline{\gamma}_j)\cdot\overline{\gamma}_j
=\gamma_j\cdot\overline{\gamma}_j$. 
we have 
$$f_{01}(D)-f_{01}(D')=
\varepsilon_1\sum_{j=2}^n
\varepsilon_{j}(t^{\gamma_{D}\cdot
\overline{\gamma}_j}-1)=
\varepsilon_1\sum_{j=2}^n\varepsilon_{j}(t^{\gamma_{j}\cdot
\overline{\gamma}_j}-1)=\varepsilon_1W_K(t).$$

\underline{$(p,q)=(1,0)$.} 
If $D'$ is obtained from $D$ 
by (a), (b), (c), or (d), 
then $D'^\#$ is obtained from $D^\#$ 
by (c), (d), (a), or (b), respectively. 
By Lemmas~2.2, 2.3, and the equation for $(p,q)=(0,1)$, it holds that 
$$f_{10}(D)-f_{10}(D')
=f_{01}(D^\#)-f_{01}(D'^\#)
=(-\varepsilon_1)\cdot W_{D^\#}(t)
=\varepsilon_1W_D(t^{-1}).$$

\underline{$(p,q)=(0,0)$.} 
It holds that 
$$f_{00}(D)-f_{00}(D')=
\varepsilon_1^2(t^{\gamma_1\cdot\gamma_1}-1)
+\sum_{j=2}^n
\varepsilon_1\varepsilon_{j}(t^{\gamma_1\cdot
\gamma_j}-1)
+\sum_{i=2}^n
\varepsilon_{i}\varepsilon_1(t^{\gamma_{i}\cdot
\gamma_1}-1). $$
For (a) and (b), we have 
$f_{00}(D)-f_{00}(D')=0$ by $\gamma_1=0$. 
For (c) and (d), 
by using $\gamma_1=\gamma_{D}$, 
it holds that 
\begin{eqnarray*}
f_{00}(D)-f_{00}(D')&=&
\varepsilon_1\sum_{j=2}^n
\varepsilon_{j}(t^{\gamma_{D}\cdot\gamma_j}-1)
+\varepsilon_1\sum_{i=2}^n
\varepsilon_{i}(t^{\gamma_{i}\cdot\gamma_{D}}-1)\\
&=&
\varepsilon_1\sum_{j=2}^n
\varepsilon_{j}(t^{-\gamma_j\cdot\overline{\gamma}_j}-1)
+\varepsilon_1\sum_{i=2}^n
\varepsilon_{i}(t^{\gamma_{i}\cdot\overline{\gamma}_i}-1)\\
&=&
\varepsilon_1W_K(t^{-1})+\varepsilon_1W_K(t)
=\varepsilon_1\overline{W}_K(t).
\end{eqnarray*}

\underline{$(p,q)=(1,1)$.} 
The proof is similar to the case $(p,q)=(1,0)$. 
For (a) and (b), $D'^\#$ is obtained from $D^\#$ by 
(c) and (d), respectively. 
By the equation for $(p,q)=(0,0)$, we have 
$$f_{11}(D)-f_{11}(D')
=f_{00}(D^\#)-f_{00}(D'^\#)
=(-\varepsilon_1)\cdot\overline{W}_{D^\#}(t)
=\varepsilon_1\overline{W}_D(t).$$
For (c) and (d), since $D'^\#$ is obtained from $D^\#$ by 
(a) and (b), respectively, 
it holds that 
$f_{11}(D)-f_{11}(D')
=f_{00}(D^\#)-f_{00}(D'^\#)=0$. 
\end{proof}

The {\it writhe} of a diagram $D$ is the sum of the signs of crossings of $D$, 
and denoted by $\omega_D=\sum_{i=1}^n \varepsilon_i$. 
We consider two kinds of Laurent polynomials 
\begin{eqnarray*}
I_D(t)&=&f_{01}(D;t)-\omega_DW_K(t) \mbox{ and }\\
I\!I_D(t)&=&f_{00}(D;t)+f_{11}(D;t)
-\omega_D\overline{W}_K(t). 
\end{eqnarray*}

\begin{theorem}\label{thm26}
The Laurent polynomials $I_D(t)$ and $I\!I_D(t)\in{\Z}[t,t^{-1}]$ 
do not depend on a particular choice of a diagram $D$ of a virtual knot $K$. 
\end{theorem}

\begin{proof}
Since the intersection numbers among 
$\gamma_i$'s and $\overline{\gamma}_i$'s 
$(1\leq i\leq n)$ do not change by a (de)stabilization, 
it is sufficient to consider the case that 
a diagram $D'$ is obtained from $D$ by 
a Reidemeister move on $\Sigma$.

Assume that $D'$ is obtained from $D$ by 
a second or third Reidemeister move on $\Sigma$. 
Since $\omega_{D'}=\omega_D$ holds, 
we have 
$I_{D'}(t)=I_D(t)$ and 
$I\!I_{D'}(t)=I\!I_D(t)$ by Lemma~\ref{lem24}. 

Assume that $D'$ is obtained from $D$ by 
a first Reidemeister move such that 
a crossing $c_1$ with the sign $\varepsilon_1$ 
is removed from $D$. 
Since it holds that $\omega_D=\omega_{D'}+\varepsilon_1$
and $f_{01}(D)-f_{01}(D')=\varepsilon_1 W_K(t)$ 
by Lemma~\ref{lem25}, 
we have 
\begin{eqnarray*}
I_D(t)&=&
f_{01}(D)-\omega_{D}W_K(t)\\
&=&
f_{01}(D')+\varepsilon_1 W_K(t)-(\omega_{D'}+\varepsilon_1)W_K(t)\\
&=&
f_{01}(D')-\omega_{D'} W_K(t)
=I_{D'}(t).
\end{eqnarray*}
Similarly, since it holds that 
$$\bigl(f_{00}(D)+f_{11}(D)\bigr)
-\bigl(f_{00}(D')+f_{11}(D')\bigr)
=\varepsilon_1\overline{W}_K(t)$$ 
by Lemma~\ref{lem25}, 
we have $I\!I_D(t)=I\!I_{D'}(t)$. 
\end{proof}

\begin{definition}\label{def27} 
The Laurent polynomials $I_D(t)$ and $I\!I_D(t)\in{\Z}[t,t^{-1}]$ 
are called the {\it first} and {\it second intersection polynomials} of a virtual knot $K$, 
and denoted by $I_K(t)$ and $I\!I_K(t)$, 
respectively. 
\end{definition}

\begin{remark}\label{rem28}
We can also consider the polynomial 
$f_{10}(D;t)-\omega_DW_K(t^{-1})$ 
which defines an invariant of $K$ 
by Lemmas~\ref{lem24} and \ref{lem25}. 
However, this is coincident with 
the first intersection polynomial $I_K(t^{-1})$. 
In fact, we have 
\begin{eqnarray*}
f_{10}(D;t)-\omega_DW_K(t^{-1})
&=&
\sum_{1\leq i,j\leq n}
\varepsilon_i(t^{\overline{\gamma}_i\cdot\gamma_j}-1)
-\omega_DW_K(t^{-1})\\
&=&\sum_{1\leq i,j\leq n} 
\varepsilon_i(t^{-\gamma_j\cdot\overline{\gamma}_i}-1)
-\omega_DW_K(t^{-1})
=I_K(t^{-1}).
\end{eqnarray*}
\end{remark}

For a Laurent polynomial $h(t)$, 
we consider an equivalence relation in ${\Z}[t,t^{-1}]$ 
such that $f(t)\equiv g(t)$ (mod~$h(t)$) holds 
if and only if $f(t)-g(t)=mh(t)$ for some $m\in{\Z}$. 
In the case of $h(t)=0$, 
this equivalence relation gives $f(t)=g(t)$ only. 
For a diagram $D$ of a virtual knot $K$ on $\Sigma$, 
we consider the equivalence class 
$f_{00}(D)$ (mod~$\overline{W}_K(t)$). 
The invariance follows by Lemmas~\ref{lem24} and \ref{lem25} 
immediately.

\begin{definition}\label{def29}
The equivalence class $f_{00}(D)$ 
\mbox{(mod~$\overline{W}_K(t)$)} 
is called the {\it third intersection polynomial} of a virtual knot $K$, 
and denoted by $I\!I\!I_K(t)$. 
\end{definition}

\begin{remark}\label{rem210}
We can also consider the equivalence class  
$f_{11}(D)$ (mod~$\overline{W}_K(t)$) 
as an invariant of $K$. 
However, since 
$I\!I_K(t)\equiv f_{00}(D)+f_{11}(D)$ (mod~$\overline{W}_K(t)$) 
holds by definition, 
we have $$f_{11}(D)\equiv 
I\!I_K(t)-I\!I\!I_K(t) \quad 
\mbox{(mod~$\overline{W}_K(t)$)}.$$
\end{remark}

A virtual knot is {\it classical} 
if it is presented by a diagram on $S^2$. 
By definition, 
the writhe polynomial $W_K(t)$ 
vanishes for any classical knot \cite{CG,K3}. 
The intersection polynomials satisfy 
the same property as follows. 

\begin{lemma}\label{lem211}
Any classical knot satisfies 
$$I_K(t)=I\!I_K(t)=I\!I\!I_K(t)=0.$$
\end{lemma}

\begin{proof} 
All intersection numbers between 
two cycles on $S^2$ are zero. 
\end{proof}


\section{Calculations}\label{sec3} 

Let $C$ be a closed, oriented curve 
on $\Sigma$ with a finite number of crossings. 
When we consider $C$ as the image of an immersion 
$S^1\rightarrow\Sigma$, 
the curve $C$ is presented by a {\it Gauss diagram} $G$ 
consisting of the circle $S^1$ equipped with chords 
each of which connects the preimage of a crossing of $C$. 
The endpoints of chords admit 
signs with respect to the orientation of $C$ 
as shown in Figure~\ref{chord}. 

\begin{figure}[htb]
\begin{center}
\includegraphics[bb = 0 0 138 37.64]{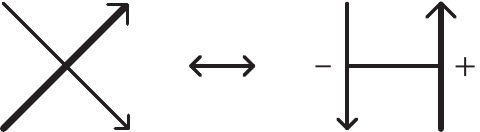}
\caption{}
\label{chord}
\end{center}
\end{figure}

The endpoints of a chord of $G$ 
divide the circle $S^1$ into two arcs. 
Let $\alpha\subset S^1$ be such an arc, 
and $P(\alpha)$ the set of endpoints 
of the chords of $G$ in the interior of $\alpha$. 
For an endpoint $x\in P(\alpha)$, 
we denote by ${\rm sgn}(x)$ the sign of $x$, 
and by $\tau(x)$ 
the other endpoint of the chord incident to $x$. 
The arc $\alpha\subset S^1$ presents 
a cycle on $\Sigma$, 
which is also denoted by $\alpha\subset\Sigma$. 
See Figure~\ref{Gauss1}.

\begin{figure}[htb]
\begin{center}
\includegraphics[bb = 0 0 305.5 81.51]{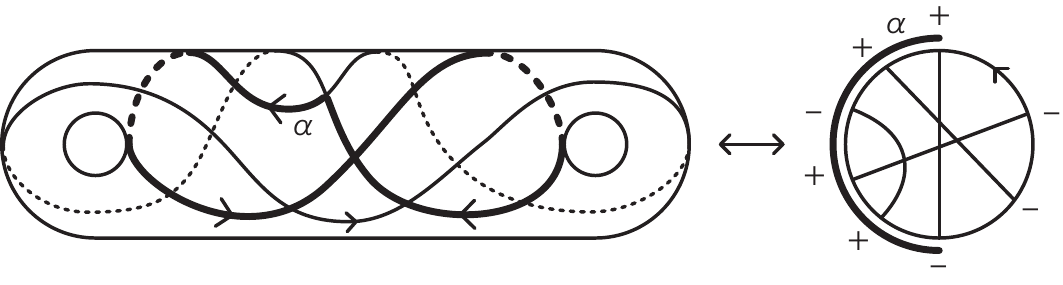}
\caption{}
\label{Gauss1}
\end{center}
\end{figure}

\begin{lemma}\label{lem31}
Let $\overline{\alpha}$ be the complementary arc of $\alpha\subset S^1$. 
Then we have 
$$\alpha\cdot\overline{\alpha}=
\sum_{x\in P(\alpha)} {\rm sgn}(x).$$
\end{lemma}

\begin{proof}
Any chord whose endpoints both lie on $\alpha$ 
does not contribute to the sum 
in the right hand side of the equation. 
Therefore it holds that 
$$\sum_{x\in P(\alpha)} {\rm sgn}(x)
=\sum_{x\in P(\alpha),\tau(x)\in P(\overline{\alpha})}
{\rm sgn}(x)=\alpha\cdot\overline{\alpha}. $$
\end{proof}

Let $\alpha$ and $\beta\subset S^1$ 
be arcs for distinct chords $a$ and $b$ of $G$, respectively. 
We consider an integer  
$$S(\alpha,\beta)=\sum_{x\in P(\alpha),\tau(x)\in P(\beta)}{\rm sgn}(x).$$
It follows by definition that 
$S(\alpha,\beta)=-S(\beta,\alpha)$. 
We say that the chords $a$ and $b$ of $G$ are {\it linked} 
if their endpoints appear on $S^1$ alternately, 
and otherwise {\it unlinked}. 
The number $S(\alpha,\beta)$ is equal to 
the sum of the signs of the endpoints indicated 
by dots as shown in Figure~\ref{calculation}.

\begin{figure}[htb]
\begin{center}
\includegraphics[bb = 0 0 314.59 69.86]{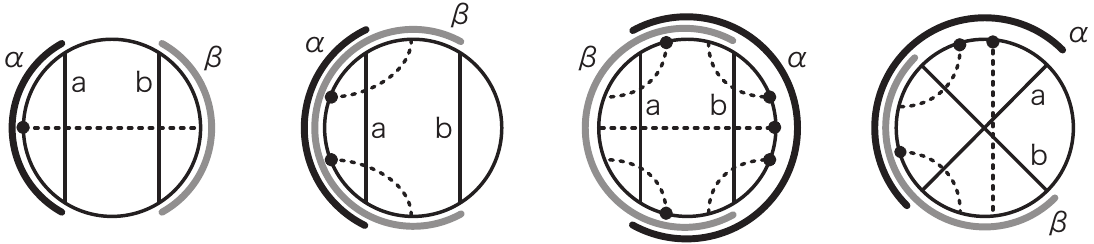}
\caption{}
\label{calculation}
\end{center}
\end{figure}

Then the intersection number $\alpha\cdot\beta$ 
of the cycles $\alpha$ and $\beta\subset\Sigma$ 
is calculated as follows.

\begin{lemma}\label{lem32}
{\rm (i)} 
If $a$ and $b$ are unlinked, 
then $\alpha\cdot \beta=S(\alpha,\beta)$. 

\begin{itemize}
\item[{\rm (ii)}] 
Assume that $a$ and $b$ are linked as shown in {\rm Figure~\ref{lem3-1}}, 
where $\varepsilon,\delta\in\{\pm\}$. 
Then it holds that 
$$\alpha\cdot\beta=S(\alpha,\beta)+\frac{1}{2}(\varepsilon+\delta) \mbox{ and }
\beta\cdot\alpha=S(\beta,\alpha)-\frac{1}{2}(\varepsilon+\delta).$$
\end{itemize}
\end{lemma}

\begin{figure}[htb]
\begin{center}
\includegraphics[bb = 0 0 70.45 83.51]{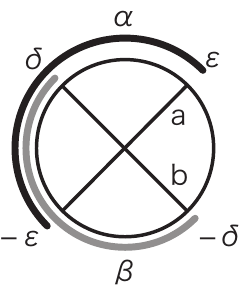}
\caption{}
\label{lem3-1}
\end{center}
\end{figure}
 
\begin{proof}
We prove two cases 
$(\varepsilon,\delta)=(+,+)$ and $(+,-)$ in (ii) 
as shown in Figure~\ref{lem3-1proof}. 
Other cases are similarly proved. 

\underline{$(\varepsilon,\delta)=(+,+)$.} 
We take a parallel copy of the curve $\alpha\subset\Sigma$ 
(or $\beta$) 
which lies on the left (or right) side of the original curve. 
See the left of Figure~\ref{lem3-1proof}. 
Then the intersections between $\alpha$ and $\beta$ 
except one point near the crossing $b$ 
correspond to the endpoints $x\in P(\alpha)$ 
with $\tau(x)\in P(\beta)$. 
Since the sign of the exceptional intersection is equal to $+$, 
we obtain 
$\alpha\cdot\beta=S(\alpha,\beta)+1$. 

\underline{$(\varepsilon,\delta)=(+,-)$.} 
Similarly to the above case, 
we consider parallel copies of $\alpha$ and $\beta$. 
In this case, there is no exceptional intersection near $b$. 
See the right of Figure~\ref{lem3-1proof}. 
Therefore we have $\alpha\cdot\beta=S(\alpha,\beta)$. 
\end{proof}

\begin{figure}[htb]
\begin{center}
\includegraphics[bb = 0 0 315.24 92.21]{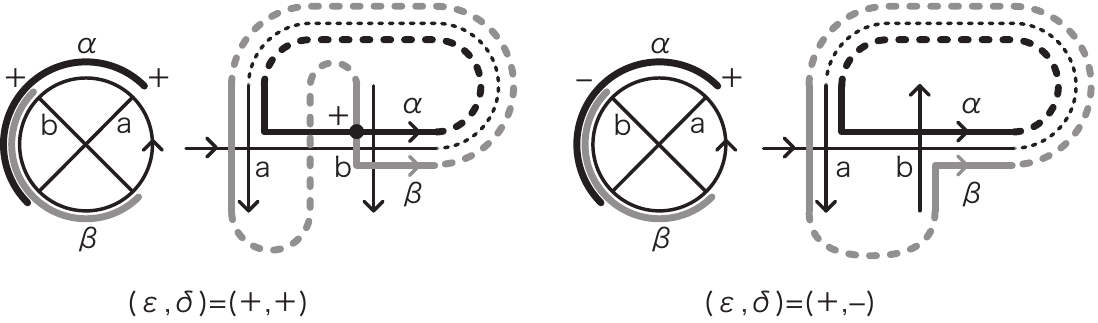}
\caption{}
\label{lem3-1proof}
\end{center}
\end{figure}

\begin{example}\label{ex33}
We consider three arcs $\alpha$, $\beta$, and $\gamma$
of the Gauss diagram as shown in Figure~\ref{Eg32}.
We have
$$S(\alpha,\beta)=-2, \ S(\alpha,\gamma)=1, \mbox{ and } S(\beta,\gamma)=0.$$
Since $\alpha$ and $\beta$ are unlinked, 
it holds that 
$\alpha\cdot\beta=S(\alpha,\beta)=-2$ 
by Lemma~\ref{lem32}(i). 
On the other hand, 
since $\alpha$ and $\gamma$, $\beta$ and $\gamma$ 
are linked, respectively, 
it holds that 
$\alpha\cdot\gamma=S(\alpha,\gamma)=1$ 
and 
$\beta\cdot\gamma=S(\beta,\gamma)+1=1$ 
by Lemma~\ref{lem32}(ii). 
\hfill$\Box$
\end{example}

\begin{center}
\begin{figure}[h]
\includegraphics[bb = 0 0 98.33 104.95]{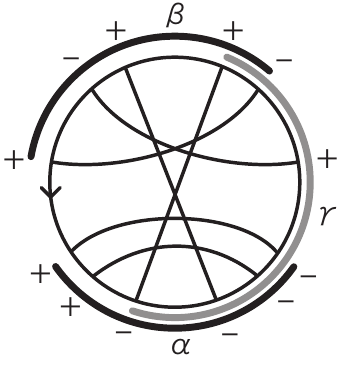}
\caption{}
\label{Eg32}
\end{figure}
\end{center}

Let $D\subset \Sigma$ be  
a diagram of a virtual knot $K$ 
with $n$ crossings $c_1,\dots,c_n$, 
and $G$ the Gauss diagram of $D$. 
We also denote by $c_i$ the chord of $G$ corresponding to 
a crossing $c_i$ of $D$. 
Each chord of $G$ is oriented 
from the over-crossing to the under-crossing, 
and equipped with the same sign as that of the corresponding crossing of $D$. 
Then we see that if the sign of a chord is equal to $\varepsilon$, 
then the initial and terminal endpoints of the chord 
have the sign $-\varepsilon$ and $\varepsilon$, respectively. 
See Figure~\ref{H1}.

\begin{figure}[htb]
\begin{center}
\includegraphics[bb = 0 0 162.59 47]{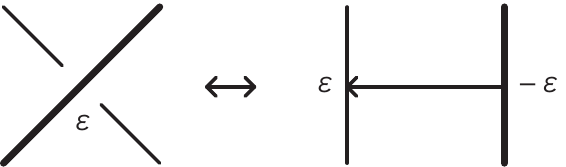}
\caption{}
\label{H1}
\end{center}
\end{figure}

The endpoints of an oriented chord $c_i$ of $G$ 
divide the circle $S^1$ into two arcs. 
The arc from the initial endpoint to the terminal 
corresponds to the cycle $\gamma_i\subset \Sigma$ 
at the crossing $c_i$, 
and the other arc corresponds to $\overline{\gamma}_i$. 
Therefore we see that 
$$\gamma_i\cdot\overline{\gamma}_i
=\gamma_i\cdot\gamma_D
=\sum_{x\in P(\gamma_i)}{\rm sgn}(x).$$ 
If two chords $c_i$ and $c_j$ are unlinked, 
then it follows by Lemma~\ref{lem32}(i) that 
$$\gamma_i\cdot\overline{\gamma}_j=S(\gamma_i,\overline{\gamma}_j), \ \gamma_i\cdot\gamma_j=S(\gamma_i,\gamma_j), \mbox{ and }
\overline{\gamma}_i\cdot\overline{\gamma}_j=
S(\overline{\gamma}_i\cdot\overline{\gamma}_j).$$

We have similar equations 
in the case that $c_i$ and $c_j$ are linked 
by Lemma~\ref{lem32}(ii).
For example, 
we consider the case as shown in Figure~\ref{linked}. 
Then we have 
\begin{eqnarray*}
\gamma_i\cdot\overline{\gamma}_j&=&
S(\gamma_i,\overline{\gamma}_j)+
(\varepsilon_i-\varepsilon_j)/2, \\
\gamma_i\cdot\gamma_j&=&
S(\gamma_i,\gamma_j)-
(\varepsilon_i+\varepsilon_j)/2, \mbox{ and }\\
\overline{\gamma}_i\cdot\overline{\gamma}_j&=&
S(\overline{\gamma}_i,\overline{\gamma}_j)
+(\varepsilon_i+\varepsilon_j)/2.
\end{eqnarray*}

\begin{figure}[htb]
\begin{center}
\includegraphics[bb = 0 0 296.07 88.59]{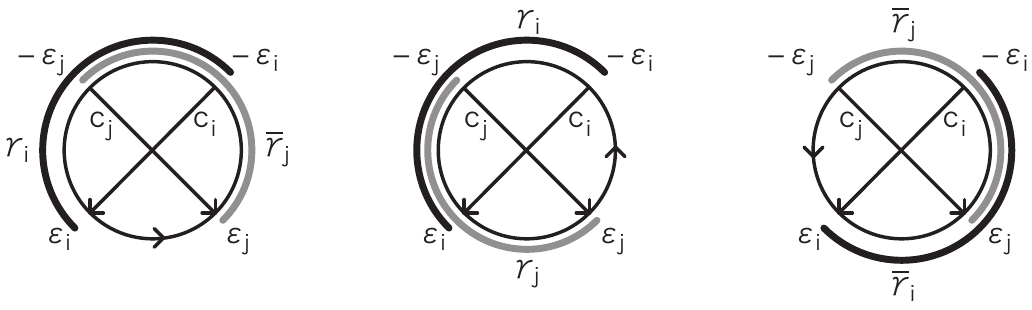}
\caption{}
\label{linked}
\end{center}
\end{figure}

Let ${\rm c}(K)$ denote the crossing number of $K$, 
which is the minimal number of crossings 
for all diagrams of $K$. 
The virtual knots up to crossing number four are given 
by Green \cite{Gr}. 
In what follows, the labels of virtual knots are due to Green's table.

\begin{example}\label{ex34}
We consider the Gauss diagram $G$ of a virtual knot $K=4.39$ 
as shown in Figure~\ref{ex3-3}. 
Table~\ref{table2}
shows the intersection numbers  
$\gamma_i\cdot \overline{\gamma}_j$, 
$\gamma_i\cdot\gamma_j$, and 
$\overline{\gamma}_i\cdot\overline{\gamma}_j$ 
for $1\leq i,j\leq 4$. 
Since we have $\varepsilon_1=\varepsilon_2=\varepsilon_3=-$ 
and $\varepsilon_4=+$, 
it holds that 
\begin{eqnarray*}
W_K(t)&=&\sum_{i=1}^n\varepsilon_i
(t^{\gamma_i\cdot\overline{\gamma}_i}-1)=
-t^3+t^2+1-t^{-1} \mbox{ and }\\
\overline{W}_K(t)&=&-t^3+t^2-t+2-t^{-1}+t^{-2}-t^{-3}.
\end{eqnarray*}
On the other hand, we have 
\begin{eqnarray*}
f_{01}(D)&=&
\sum_{1\leq i,j\leq n} \varepsilon_i\varepsilon_j(t^{\gamma_i\cdot\overline{\gamma}_j}-1)
=2t^2-2t-2+2t^{-1}, \\
f_{00}(D)&=&
\sum_{1\leq i,j\leq n} \varepsilon_i\varepsilon_j(t^{\gamma_i\cdot\gamma_j}-1)
=t^3-t^2-t^{-2}+t^{-3}, \mbox{ and }\\
f_{11}(D)&=&
\sum_{1\leq i,j\leq n} \varepsilon_i\varepsilon_j(t^{\overline{\gamma}_i\cdot
\overline{\gamma}_j}-1)=t^2-t-t^{-1}+t^{-2}. 
\end{eqnarray*}
Since $\omega_D=-2$ holds, 
we obtain 
\begin{eqnarray*}
I_K(t)&=&-2t^3+4t^2-2t, \\
I\!I_K(t)&=&-t^3+2t^2-3t+4-3t^{-1}+2t^{-2}-t^{-3}, \mbox{ and }\\
I\!I\!I_K(t)&\equiv&-t+2-t^{-1} 
\quad \mbox{(mod~$-t^3+t^2-t+2-t^{-1}+t^{-2}-t^{-3}$)}.
\end{eqnarray*}
\hfill$\Box$
\end{example}

\begin{figure}[htb]
\begin{center}
\includegraphics[bb = 0 0 78.47 87.95]{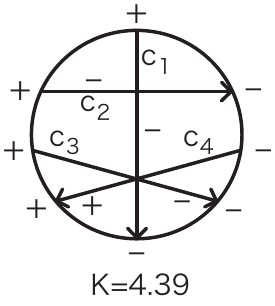}
\caption{}
\label{ex3-3}
\end{center}
\end{figure}

\begin{table}[htb]
\begin{center}
\begin{tabular}{|r|r|r|r|r|}
\hline
 & $\overline{\gamma}_1$ & $\overline{\gamma}_2$ & $\overline{\gamma}_3$ & 
$\overline{\gamma}_4$ \\
\hline
$\gamma_1$ & $3$ & $0$ & $2$ & $2$ \\
\hline
$\gamma_2$ & $2$ & $-1$ & $0$ & $1$ \\
\hline
$\gamma_3$ & $1$ & $-1$ & $0$ & $1$ \\
\hline
$\gamma_4$ & $3$ & $0$ & $1$ & $2$ \\
\hline
\end{tabular}
\quad
\begin{tabular}{|r|r|r|r|r|}
\hline
 & $\gamma_1$ & $\gamma_2$ & $\gamma_3$ & 
$\gamma_4$ \\
\hline
$\gamma_1$ & $0$ & $3$ & $1$ & $1$ \\
\hline
$\gamma_2$ & $-3$ & $0$ & $-1$ & $-2$ \\
\hline
$\gamma_3$ & $-1$ & $1$ & $0$ & $-1$ \\
\hline
$\gamma_4$ & $-1$ & $2$ & $1$ & $0$ \\
\hline
\end{tabular}
\quad
\begin{tabular}{|r|r|r|r|r|}
\hline
 & $\overline{\gamma}_1$ & $\overline{\gamma}_2$ & $\overline{\gamma}_3$ & 
$\overline{\gamma}_4$ \\
\hline
$\overline{\gamma}_1$ & $0$ & $-1$ & $-2$ & $0$ \\
\hline
$\overline{\gamma}_2$ & $1$ & $0$ & $0$ & $1$ \\
\hline
$\overline{\gamma}_3$ & $2$ & $0$ & $0$ & $1$ \\
\hline
$\overline{\gamma}_4$ & $0$ & $-1$ & $-1$ & $0$ \\
\hline 
\end{tabular}
\end{center}
\caption{}
\label{table2}
\end{table}

\begin{theorem}\label{thm35}
For the virtual knots $K$ with ${\rm c}(K)\leq 4$, 
the intersection polynomials $I_K(t)$, $I\!I_K(t)$, 
and $I\!I\!I_K(t)$ 
are given in {\rm Appendices \ref{appA}} and {\rm \ref{appB}}. 
\hfill$\Box$
\end{theorem}

By observing the calculations in Appendices \ref{appA} and \ref{appB}, 
we see that the writhe polynomial and 
the intersection polynomials are 
independent of each other in the following sense.

\begin{proposition}\label{prop36}
There are four pairs of virtual knots 
$K_i$ and $K_i'$ $(i=1,2,3,4)$ 
which satisfy the following. 
\begin{itemize}
\item[{\rm (i)}] 
$W_{K_1}(t)\ne W_{K_1'}(t)$ and 
$X_{K_1}(t)= X_{K_1'}(t)$ $(X=I, I\!I, I\!I\!I)$. 
\item[{\rm (ii)}] 
$I_{K_2}(t)\ne I_{K_2'}(t)$ and 
$X_{K_2}(t)= X_{K_2'}(t)$ $(X=W, I\!I, I\!I\!I)$. 
\item[{\rm (iii)}] 
$I\!I_{K_3}(t)\ne I\!I_{K_3'}(t)$ and 
$X_{K_3}(t)= X_{K_3'}(t)$ $(X=W, I, I\!I\!I)$. 
\item[{\rm (iv)}] 
$I\!I\!I_{K_4}(t)\ne I\!I\!I_{K_4'}(t)$ and 
$X_{K_4}(t)= X_{K_4'}(t)$ $(X=W, I, I\!I)$. 
\end{itemize}
\end{proposition}

\begin{proof}
(i) For the virtual knots $K_1=4.36$ and $K_1'=4.65$, 
it holds that 
$$W_{K_1}(t)=t^2-2+t^{-2}\mbox{ and }W_{K_1'}(t)=-t^2+2-t^{-2}.$$
On the other hand, we have 
\begin{eqnarray*}
I_{K_1}(t)&=&I_{K_1'}(t)=-t^2+2-t^{-2}, \\
I\!I_{K_1}(t)&=&I\!I_{K_1'}(t)=-2t^2+4-2t^{-2}, \mbox{ and}\\
I\!I\!I_{K_1}(t)&=&I\!I\!I_{K_1'}(t) \equiv t^2-2+t^{-2} \quad 
\mbox{(mod~$2t^2-4+2t^{-2}$)}. 
\end{eqnarray*}

(ii) For the trivial knot $K_2=O$ and the virtual knot $K_2'=4.16$, 
it holds that 
$$I_{K_2}(t)=0 \mbox{ and }I_{K_2'}(t)=-t^2+3t-3+t^{-1}. $$
On the other hand, we have 
\begin{eqnarray*}
W_{K_2}(t)&=&W_{K_2'}(t)=0, \\ 
I\!I_{K_2}(t)&=&I\!I_{K_2'}(t)=0, \mbox{ and} \\
I\!I\!I_{K_1}(t)&=&I\!I\!I_{K_1'}(t)=0. 
\end{eqnarray*} 

(iii) For the virtual knots $K_3=3.2$ and $K_3'=4.33$, 
it holds that 
$$I\!I_{K_3}(t)=-2t+4-2t^{-1}\mbox{ and }
I\!I_{K_3'}(t)=t^2-6t+10-6t^{-1}+t^{-2}.$$
On the other hand, we have 
\begin{eqnarray*}
W_{K_3}(t)&=&W_{K_3'}(t)=-t+2-t^{-1}, \\
I_{K_3}(t)&=&I_{K_3'}(t)=-t+2-t^{-1}, \mbox{ and}\\
I\!I\!I_{K_1}(t)&=&I\!I\!I_{K_1'}(t)\equiv t-2+t^{-1} \quad 
\mbox{(mod~$2t-4+2t^{-1}$)}. 
\end{eqnarray*}

(iv) For the virtual knots $K_4=O$ and $K_4'=4.13$, 
it holds that 
$$I\!I\!I_{K_4}(t)=0 \mbox{ and }
I\!I\!I_{K_4'}(t)=2t-4+2t^{-1}.$$
On the other hand, we have 
\begin{eqnarray*}
W_{K_4}(t)&=&W_{K_4'}(t)=0, \\
I_{K_4}(t)&=&I_{K_4'}(t)=0, \mbox{ and}\\
I\!I_{K_4}(t)&=&I\!I_{K_4'}(t)=0.
\end{eqnarray*}

The Gauss diagrams of the knots 
are illustrated in Figure~\ref{cor3-5}. 
\end{proof}

\begin{figure}[htb]
\begin{center}
\includegraphics[bb = 0 0 354.94 70.32]{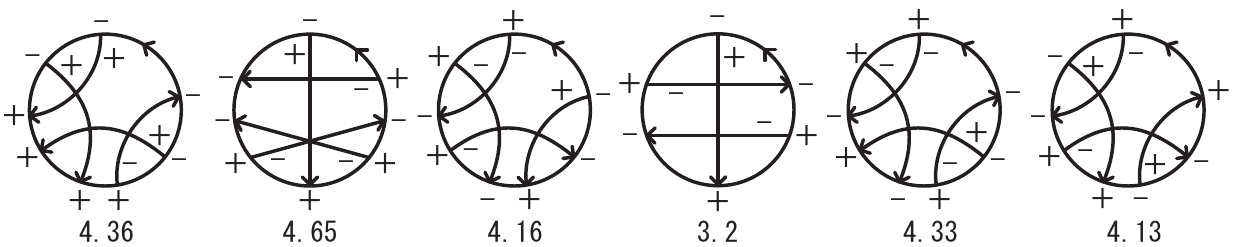}
\caption{}
\label{cor3-5}
\end{center}
\end{figure}

\begin{remark}\label{rem37} 
For a virtual knot $K$, 
Silver and Williams \cite{SW} define a sequence of 
Alexander polynomials $\Delta_i(K)$ 
as an extension of the classical Alexander polynomial. 
The zero-th Alexander polynomial $\Delta_0(K)(u,v)$ 
is also defined by Sawollek \cite{Saw}. 
Mellor \cite{Mel} proves that the writhe polynomial 
is obtained from $\Delta_0(K)(u,v)=(1-uv)\widetilde{\Delta}_0(K)(u,v)$ 
by the equation 
$$W_K(t)=-\widetilde{\Delta}_0(K)(t,t^{-1}).$$
It is natural to ask whether the intersection polynomials 
are also obtained from $\Delta_i(K)$. 
However this does not hold generally; 
in fact, for the virtual knot $K=4.8$, 
we have $\Delta_0(K)=0$ and $\Delta_i(K)=1$ $(i\geq 1)$ 
which are coincident with those of the trivial knot. 
On the other hand, it holds that 
the intersection polynomials $I_K(t)$, $I\!I_K(t)$, and ${\III}_K(t)$ are all non-trivial. 
\end{remark}


\section{Symmetries and crossing numbers}\label{sec4} 

For a diagram $D$ on $\Sigma$ of a virtual knot $K$, 
let $-D$ be the diagram 
by reversing the orientation of $D$, 
$D^\#$ the one 
by changing over/under-information at every crossing of $D$, 
and $D^*$ the one obtained 
by an orientation-reversing homeomorphism of $\Sigma$. 
The virtual knots presented by 
$-D$, $D^\#$, and $D^*$
are called the {\it reverse}, the {\it vertical mirror image}, 
and the {\it horizontal mirror image} of $K$, 
and denoted by $-K$, $K^\#$, and $K^*$, respectively. 
The writhe polynomials of these knots 
are given as follows. 

\begin{lemma}[\cite{CG,K3,ST}]\label{lem51}
Any virtual knot $K$ satisfies 
$$W_{-K}(t)=W_K(t^{-1}) \mbox{ and }
W_{K^\#}(t)=W_{K^*}(t)=-W_K(t^{-1}).$$
Therefore we have 
$$\overline{W}_{-K}(t)=\overline{W}_K(t) \mbox{ and }
\overline{W}_{K^\#}(t)=\overline{W}_{K^*}(t)
=-\overline{W}_K(t).$$
\end{lemma}

The intersection polynomials of 
$-K$, $K^\#$, and $K^{*}$ are 
given as follows. 

\begin{lemma}\label{lem52}
For a virtual knot $K$, 
we have the following. 
\begin{itemize}
\item[{\rm (i)}] 
$I_{-K}(t)=I_{K^\#}(t)=I_{K^*}(t)=I_K(t^{-1})$. 
\item[{\rm (ii)}] 
$I\!I_{-K}(t)=I\!I_{K^\#}(t)=I\!I_{K^*}(t)=I\!I_K(t)$. 
\item[{\rm (iii)}] 
$I\!I\!I_{-K}(t)=I\!I\!I_{K^\#}(t)=I\!I_K(t)-I\!I\!I_K(t)$ 
and 
$I\!I\!I_{K^{*}}(t)=I\!I\!I_K(t)$. 
\end{itemize}
\end{lemma}

\begin{proof}
\underline{$I_{-K}(t)$, $I\!I_{-K}(t)$, and $I\!I\!I_{-K}(t)$.} 
Let $c_i'$ be the crossing of $-D$ corresponding to $c_i$, 
$\gamma_i'$ the cycle at $c_i'$ on $\Sigma$, 
and $\varepsilon_i'$ the sign of $c_i'$ $(1\leq i\leq n)$. 
Then it holds that 
$\gamma_i'=-\overline{\gamma}_i$ 
and $\varepsilon_i'=\varepsilon_i$. 
By definition, we have 
\begin{eqnarray*}
f_{01}(-D;t)&=&
\sum_{1\leq i,j\leq n}\varepsilon_i'\varepsilon_j'
(t^{\gamma_i'\cdot\overline{\gamma}_j'}-1)
=\sum_{1\leq i,j\leq n}\varepsilon_i\varepsilon_j
(t^{\overline{\gamma}_i\cdot\gamma_j}-1)\\
&=&\sum_{1\leq i,j\leq n}\varepsilon_i\varepsilon_j
(t^{-\gamma_j\cdot\overline{\gamma}_i}-1)
=f_{01}(D;t^{-1}), \\ 
f_{00}(-D;t)&=&
\sum_{1\leq i,j\leq n}\varepsilon_i\varepsilon_j
(t^{\overline{\gamma}_i\cdot\overline{\gamma}_j}-1)
=f_{11}(D;t), \mbox{ and }\\
f_{11}(-D;t)&=&
\sum_{1\leq i,j\leq n}\varepsilon_i\varepsilon_j
(t^{\gamma_i\cdot\gamma_j}-1)
=f_{00}(D;t).
\end{eqnarray*}

Since $\omega_{-D}=\omega_D$ holds, 
we have 
\begin{eqnarray*}
I_{-K}(t)&=&
f_{01}(-D;t)-\omega_{-D}W_{-D}(t)\\
&=&f_{01}(D;t^{-1})-\omega_DW_D(t^{-1}) 
=I_K(t^{-1}), \\
I\!I_{-K}(t)&=&
f_{00}(-D;t)+f_{11}(-D;t)-\omega_{-D}\bigl(W_{-D}(t)+W_{-D}(t^{-1})\bigr)\\
&=&
f_{11}(D;t)+f_{00}(D;t)-\omega_D
\bigl(W_D(t^{-1})+W_D(t)\bigr)
=I\!I_K(t) \mbox{ and }\\
I\!I\!I_{-K}(t)&\equiv&
f_{00}(-D;t)=f_{11}(D;t)
\equiv I\!I_K(t)-f_{00}(D;t)\equiv I\!I_K(t)-I\!I\!I_K(t). 
\end{eqnarray*}

\underline{$I_{K^\#}(t)$, $I\!I_{K^\#}(t)$, and 
$I\!I\!I_{K^\#}(t)$.} 
We use the notations in the proof of 
Lemma~\ref{lem22}. 
By the lemma, 
we have 
$f_{00}(D^\#;t)=f_{11}(D;t)$ 
and hence $f_{11}(D^\#;t)=f_{00}(D;t)$. 
Furthermore it holds that 
$$f_{01}(D^\#;t)=
\sum_{1\leq i,j\leq n}(-\varepsilon_i)(-\varepsilon_j)
(t^{\overline{\gamma}_i\cdot\gamma_j}-1)
=f_{01}(D;t^{-1}).$$
Since it holds that 
$\omega_{D^\#}=-\omega_D$ and 
$W_{D^\#}(t)=-W_D(t^{-1})$, 
we have 
\begin{eqnarray*}
I_{K^\#}(t)&=&
f_{01}(D^\#;t)-\omega_{D^\#}W_{D^\#}(t)\\
&=&f_{01}(D;t^{-1})-\omega_DW_D(t^{-1}) 
=I_K(t^{-1}) \\
I\!I_{K^\#}(t)&=&
f_{00}(D^\#;t)+f_{11}(D^\#;t)-\omega_{D^\#}\bigl(W_{D^\#}(t)+W_{D^\#}(t^{-1})\bigr)\\
&=&
f_{11}(D;t)+f_{00}(D;t)-\omega_D
\bigl(W_D(t^{-1})+W_D(t)\bigr)
=I\!I_K(t), \mbox{ and }\\
I\!I\!I_{K^\#}(t)&\equiv&
f_{00}(D^\#;t)=f_{11}(D;t) 
\equiv I\!I_K(t)-f_{00}(D;t)\equiv 
I\!I_K(t)-I\!I\!I_K(t). 
\end{eqnarray*}

\underline{$I_{K^*}(t)$, $I\!I_{K^*}(t)$, 
and $I\!I\!I_{K^*}(t)$.} 
Let $c_i^*$ be the crossing of $D^*$ corresponding to $c_i$, 
$\gamma_i^*$ the cycle at $c_i^*$ on $\Sigma$, 
and $\varepsilon_i^*$ the sign of $c_i^*$ $(1\leq i\leq n)$. 
Then it holds that 
$$\gamma_i^*\cdot\overline{\gamma}_j^*
=-\gamma_i\cdot\overline{\gamma}_j, \ 
\gamma_i^*\cdot\gamma_j^*
=-\gamma_i\cdot\gamma_j, \ 
\overline{\gamma}_i^* \cdot\overline{\gamma}_j^*
=-\overline{\gamma}_i\cdot\overline{\gamma}_j$$
and $\varepsilon_i^*=-\varepsilon_i$. 
Since $f_{00}(D;t)$ and $f_{11}(D;t)$ are reciprocal,
we have 
$$f_{01}(D^*;t)=f_{01}(D;t^{-1}), \ 
f_{00}(D^*;t)=f_{00}(D;t), \mbox{ and }
f_{11}(D^*;t)=f_{11}(D;t).$$
Since it holds that 
$\omega_{D^*}=-\omega_D$ and 
$W_{D^*}(t)=-W_D(t^{-1})$, 
we have 
\begin{eqnarray*}
I_{K^*}(t)&=&
f_{01}(D^*;t)-\omega_{D^*}W_{D^*}(t)\\
&=&f_{01}(D;t^{-1})-\omega_DW_D(t^{-1}) 
=I_K(t^{-1}),  \\
I\!I_{K^*}(t)&=&
f_{00}(D^*;t)+f_{11}(D^*;t)
-\omega_{D^*}\bigl(W_{D^*}(t)+W_{D^*}(t^{-1})\bigr)\\
&=&
f_{00}(D;t)+f_{11}(D;t)-\omega_D
\bigl(W_D(t^{-1})+W_D(t)\bigr)
=I\!I_K(t), \mbox{ and }\\
I\!I\!I_{K^*}(t)&\equiv&
f_{00}(D^*;t) = 
f_{00}(D;t)\equiv I\!I\!I_K(t). 
\end{eqnarray*}
\end{proof}

\begin{proposition}\label{prop53}
If a virtual knot $K$ satisfies 
\begin{itemize}
\item[{\rm (i)}] 
$2I\!I\!I_K(t)\not\equiv I\!I_K(t)$ 
{\rm (mod~$\overline{W}_K(t)$)}, 
\item[{\rm (ii)}] 
$W_K(t)\ne W_K(t^{-1})$, and 
\item[{\rm (iii)}] 
$W_K(t)\ne -W_K(t^{-1})$ or 
$I_K(t)\ne I_K(t^{-1})$, 
\end{itemize}
then the eight virtual knots 
$$K,-K,K^\#,K^*, 
-K^\#, -K^*, K^{\#*}, 
\mbox{ and } -K^{\#*}$$ 
are mutually distinct. 
\end{proposition}

\begin{proof}
By Lemma~\ref{lem52}(iii), 
the virtual knots $K, K^*, -K^\#$, and $-K^{\#*}$  
have the same third intersection polynomial $I\!I\!I_K(t)$, 
and $-K, -K^{*}, K^\#$, and $K^{\#*}$ 
have $I\!I_K(t)-I\!I\!I_K(t)$. 
Therefore, by the condition (i), 
it holds that 
$$\{K, K^*, -K^\#, -K^{\#*}\}
\cap\{-K, -K^{*}, K^\#,K^{\#*}\}=\emptyset.$$

Furthermore, by Lemma~\ref{lem51}, 
the first four virtual knots $K, K^*, -K^\#$, and $-K^{\#*}$ 
have the writhe polynomials 
$W_K(t)$, $-W_K(t^{-1})$, $-W_K(t)$, and $W_K(t^{-1})$, 
respectively. 
Since $W_K(t)\ne 0$ follows by the condition (ii), 
we have 
$$\{K,K^{*}\}\cap \{-K^\#,-K^{\#*}\}=\emptyset.$$

Finally, each of the pairs $K$ and $K^{*}$, 
and  $-K^\#$ and $-K^{\#*}$ can be distinguished 
by the condition (iii) and Lemma~\ref{lem52}(i). 
We can prove that the latter four virtual knots 
$-K, -K^{*}, K^\#$, and $K^{\#*}$ are 
mutually distinct similarly. 
\end{proof}

\begin{theorem}\label{thm54}
There are infinitely many virtual knots $K$ such that 
$$K,-K,K^\#,K^*, 
-K^\#, -K^*, K^{\#*}, 
\mbox{ and } -K^{\#*}$$ 
are mutually distinct. 
\end{theorem} 

\begin{proof}
Let $K_n$ $(n\geq 1)$ be the virtual knot presented 
by the Gauss diagram $G_n$ as shown in Figure~\ref{th54}.
It holds that 
\begin{eqnarray*}
W_{K_n}(t)&=&t^{n+1}-(n+1)t +(n+1)t^{-1}-t^{-n-1},\\
I_{K_n}(t)&=&-t^{n+2}-nt^{n+1}+t-(2n-1)
-(n+1)t^{-n-1}+2\sum_{i=1}^{n}(t^i+t^{-i}),\\
I\!I_{K_n}(t)&=&
-(t^{2n+2}+t^{-2n-2})-(t^{n+1}+t^{-n-1})-(2n+1)(t^{n}+t^{-n})\\
&&-n(n+1)(t^2+t^{-2})-(2n+3)(t+t^{-1})+2(n^2+n+2)\\
&&+4 \sum_{i=1}^{n+1}(t^i+t^{-i}), \mbox{ and }\\
I\!I\!I_{K_n}(t)&=&
-(t^{n+1}+t^{-n-1})-(n+1)(t+t^{-1})
+2\sum_{i=1}^{n+1}(t^i+t^{-i})-2n, 
\end{eqnarray*}
where we have $\overline{W}_{K_n}(t)=0$. 
Since these invariants of $K_n$ satisfy 
the conditions (i), (ii), and (iii) $I_{K_n}(t)\ne I_{K_n}(t^{-1})$ 
in Proposition~\ref{prop53}, 
the eight kinds of virtual knots associated with $K_n$ 
are mutually distinct. 
Furthermore $K_n\ne K_m$ $(n\ne m)$ holds 
by ${\rm deg}W_{K_n}(t)=n+1$. 
We remark that $K_n$ satisfies $W_{K_n}(t)=-W_{K_n}(t^{-1})$. 
\end{proof}

\begin{figure}[htb]
\begin{center}
\includegraphics[bb = 0 0 143.97 108.38]{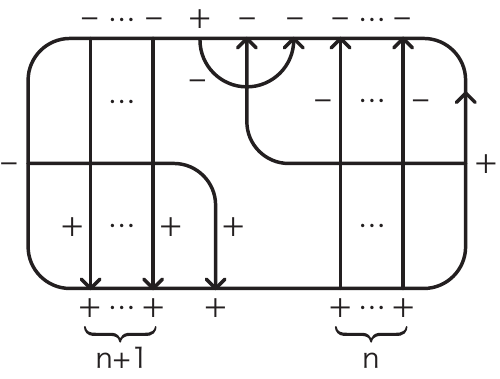}
\caption{}
\label{th54}
\end{center}
\end{figure}

\begin{example}\label{ex55}
We can construct an infinite family of virtual knots $K$ 
satisfying the conditions (i), (ii), and (iii) $W_K(t)\ne -W_K(t^{-1})$ 
in Proposition~\ref{prop53}. 

Let $K_n'$ $(n\geq3)$ be the virtual knot presented 
by the Gauss diagram $G_n'$ as shown in Figure \ref{Eg55}. 
We have 
\begin{eqnarray*}
W_{K_n'}(t)&=&-t^{n-1}+nt-n+t^{-1},\\
I_{K_n'}(t)&=& -t^n+(n-2)t^{n-1}-2\sum_{i=1}^{n-2}t^{i}+(2n-1)-nt^{-1}, \\
I\!I_{K_n'}(t)&=& -(t^{n}+t^{-n})+(n-1)(t^{n-1}+t^{-n+1})\\
& & -2\sum_{i=1}^{n-2}(t^{i}+t^{-i})-(n^2-n+1)(t+t^{-1})+2n^2-2, \mbox{ and}\\
I\!I\!I_{K_n'}(t)&\equiv& 
-\sum_{i=1}^{n-2}(t^i+t^{-i})-(t+t^{-1})+2n-2
\quad \mbox{(mod~$\overline{W}_{K_n'}(t)$)}, 
\end{eqnarray*}
where 
$\overline{W}_{K_n'}(t)=-(t^{n-1}+t^{-n+1})+(n+1)(t+t^{-1})-2n$. 
Since these invariants of $K_n'$ satisfy 
the conditions (i), (ii), and (iii) $W_{K_n'}(t)\ne -W_{K_n'}(t^{-1})$ 
in Proposition~\ref{prop53}, 
the eight kinds of virtual knots associated with $K_n'$ 
are mutually distinct. 
Furthermore 
$K_n'\ne K_m'$ $(n\ne m)$ holds
by ${\rm deg}W_{K_n'}(t)=n-1$. 
We remark that $K_n'$ satisfies (iii) $I_{K_n'}(t)\neq I_{K_n'}(t^{-1})$
in Proposition~\ref{prop53}.
\hfill$\Box$
\end{example}

\begin{figure}[htb]
\begin{center}
\includegraphics[bb = 0 0 107.61 87.45]{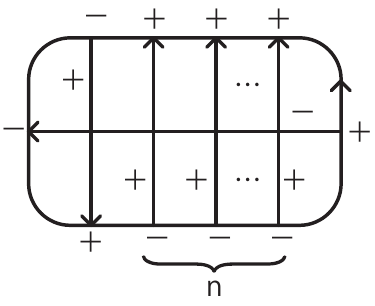}
\caption{}
\label{Eg55}
\end{center}
\end{figure}


\section{Real and virtual crossing numbers}\label{sec5}

For a Laurent polynomial $f(t)$, 
let ${\rm deg}f(t)$ 
denote the maximal degree of $f(t)$. 
The writhe polynomial $W_K(t)$ gives a lower bound 
of the crossing number ${\rm c}(K)$ as follows. 

\begin{lemma}[\cite{ST}]\label{lem56} 
Any non-trivial virtual knot $K$ 
satisfies ${\rm c}(K)\geq {\rm deg}W_K(t)+1$. 
\end{lemma}

We remark that the minimal degree of $W_K(t)$ 
also gives a lower bound of ${\rm c}(K)={\rm c}(-K)$ 
by the equation to $W_K(t^{-1})=W_{-K}(t)$. 
The {\it span} of $W_K(t)$ is 
the difference of the maximal and minimal degrees of $W_K(t)$, 
and denoted by ${\rm span} W_K(t)$. 
Then Lemma~\ref{lem56} induces a weaker inequation 
$${\rm c}(K)\geq\frac{1}{2}{\rm span}W_K(t)+1$$ 
immediately. 
The intersection polynomials also gives 
lower bounds of ${\rm c}(K)$ as follows. 
Here, ${\rm deg}{\III}_K(t)$ denotes 
the maximal number of ${\rm deg}f(t)$ 
for all $f(t)$ with $f(t)\equiv I\!I\!I_K(t)$ 
$({\rm mod}~\overline{W}_K(t))$. 

\begin{proposition}\label{prop57}
Let $K$ be a non-trivial virtual knot. 
\begin{itemize}
\item[{\rm (i)}] 
${\rm c}(K)\geq {\rm deg}I_K(t)+1$. 
\item[{\rm (ii)}] 
${\rm c}(K)\geq {\rm deg}I\!I_K(t)+1$. 
\item[{\rm (iii)}] 
${\rm c}(K)\geq {\rm deg}{\III}_K(t)+1$. 
\end{itemize}
\end{proposition}

\begin{proof}
Assume that a diagram $D$ of $K$ 
satisfies $c(D)={\rm c}(K)$. 
Since $K$ is non-trivial, 
it holds that $c(D)\geq 2$. 
For $(\alpha,\beta)=(\gamma_i,\overline{\gamma}_j)$, 
$(\gamma_i,\gamma_j)$, and 
$(\overline{\gamma}_i,\overline{\gamma}_j)$ 
with $i\ne j$, 
the intersection number $\alpha\cdot\beta$ 
is equal to 
$S(\alpha,\beta)-1$, $S(\alpha,\beta)$, 
or $S(\alpha,\beta)+1$ by Lemma~\ref{lem32} 
so that we obtain 
$\alpha\cdot\beta\leq S(\alpha,\beta)+1$. 

Since there are $c(D)-2$ chords other than $c_i$ and $c_j$
in the Gauss diagram of $D$, 
we have $S(\alpha,\beta)\leq c(D)-2$. 
Therefore it holds that 
$${\rm deg}f_{pq}(D)\leq c(D)-1={\rm c}(K)-1$$ 
for $(p,q)=(0,1), (0,0)$, and $(1,1)$. 
Since ${\rm deg}W_K(t)\leq {\rm c}(K)-1$, 
we have the conclusion. 
\end{proof}

As well as a diagram on $\Sigma$ or 
a Gauss diagram, 
a virtual knot is also presented by 
a {\it virtual diagram in ${\R}^2$}. 
It is an immersed circle in ${\R}^2$ 
with real and virtual crossings \cite{K1}. 
Here, the real crossings correspond to 
the crossings on $\Sigma$, 
and the virtual crossings are surrounded by 
small circles. 
The {\it virtual crossing number} of a virtual knot $K$ 
is the minimal number of virtual crossings 
for all virtual diagrams of $K$, 
and denoted by ${\rm vc}(K)$.

The intersection polynomials 
are also calculated from a virtual diagram. 
For example, we consider the virtual diagram 
with three real crossings $c_1$, $c_2$, and $c_3$ 
and two virtual crossings 
as shown in the leftmost of Figure~\ref{3-4-ori}, 
which presents the virtual knot $K=3.4$. 
To calculate $\gamma_1\cdot\overline{\gamma}_2$, 
we draw the curves $\gamma_1$ and $\overline{\gamma}_2$ 
equipped with virtual crossings, 
and then take the sum of signs of two intersections 
with ignoring virtual crossings 
to obtain $\gamma_1\cdot\overline{\gamma}_2=2$. 
See the second from the left in the figure. 
Similarly we obtain 
$\gamma_1\cdot\gamma_2=0$ and 
$\overline{\gamma}_1\cdot\overline{\gamma}_2=-1$ 
as shown in the third and fourth from the left.

\begin{figure}[htb]
\begin{center}
\includegraphics[bb = 0 0 310.63 103.13]{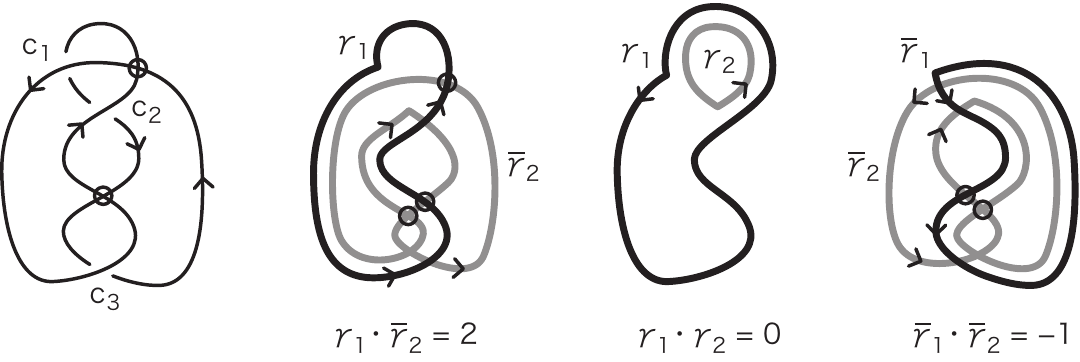}
\caption{}
\label{3-4-ori}
\end{center}
\end{figure}

The writhe polynomial $W_K(t)$ gives a lower bound 
of the virtual crossing number ${\rm vc}(K)$ as follows.

\begin{lemma}[\cite{ST}]\label{lem58} 
Any non-trivial virtual knot $K$ 
satisfies ${\rm vc}(K)\geq {\rm deg}W_K(t)$. 
\end{lemma}

We remark that 
Lemma~\ref{lem58} induces a weaker inequation 
$${\rm vc}(K)\geq\frac{1}{2}{\rm span}W_K(t)$$ 
immediately, 
which is proved in \cite{Mel}. 
The intersection polynomials also gives 
lower bounds of ${\rm vc}(K)$ as follows. 

\begin{proposition}\label{prop59}
Let $K$ be a virtual knot. 
\begin{itemize}
\item[{\rm (i)}] 
${\rm vc}(K)\geq {\rm deg}I_K(t)$. 
\item[{\rm (ii)}] 
${\rm vc}(K)\geq {\rm deg}I\!I_K(t)$. 
\item[{\rm (iii)}] 
${\rm vc}(K)\geq {\rm deg}{\III}_K(t)$. 
\end{itemize}
\end{proposition}

\begin{proof}
Let $D$ be a virtual diagram of $K$ in ${\R}^2$, 
and $\alpha$ and $\beta$ cycles on $D$ 
with corners at (possibly the same) real crossings of $D$. 
By a slight perturbation of $\beta$ if necessary, 
we may assume that $\alpha$ and $\beta$ 
intersect in a finite number of double points 
near real and virtual crossings of $D$ 
as explained as above. 
By Lemma~\ref{lem58}, it is sufficient to prove that 
if the intersection number 
restricted to the real crossings 
between $\alpha$ and $\beta$ in ${\R}^2$ 
is equal to $n$, 
then the number of virtual crossings of $D$ 
is greater than or equal to $|n|$. 

Since the total intersection number between 
$\alpha$ and $\beta$ in ${\R}^2$ is equal to zero, 
the intersection number 
restricted to the virtual crossings between $\alpha$ and $\beta$ 
is equal to $-n$. 

Let $v$ be a virtual crossing of $D$ where 
two short paths $\lambda$ and $\lambda'\subset D$ intersect. 
If $\alpha$ and $\beta$ intersect in virtual crossings near $v$, 
then 
there are four cases as follows. 
\begin{itemize}
\item[(i)] 
$\alpha\supset \lambda$, $\alpha\not\supset\lambda'$, 
and $\beta\not\supset \lambda$, $\beta\supset\lambda'$. 
\item[(ii)] 
$\alpha\supset \lambda,\lambda'$, 
and $\beta\not\supset \lambda$, $\beta\supset\lambda'$. 
\item[(iii)] 
$\alpha\supset \lambda$, $\alpha\not\supset\lambda'$, 
and $\beta\supset\lambda,\lambda'$. 
\item[(iv)] 
$\alpha\supset \lambda,\lambda'$ and 
$\beta\supset\lambda,\lambda'$. 
\end{itemize}
See Figure~\ref{virtual}. 

\begin{center}
\begin{figure}[h]
\includegraphics[bb = 0 0 312.92 67.6]{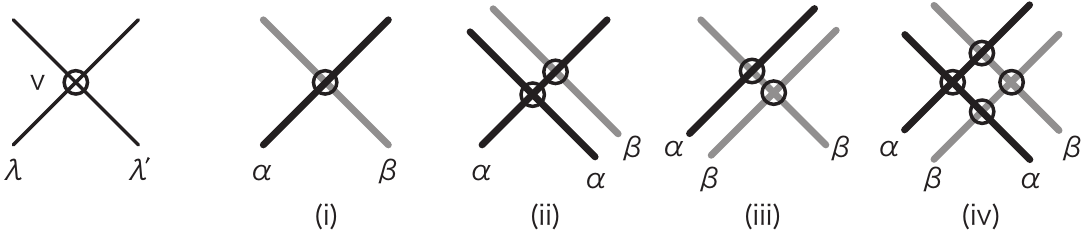}
\caption{}
\label{virtual}
\end{figure}
\end{center}

In the case (iv), the pair of virtual crossings between $\alpha$ and $\beta$ 
does not contribute to the intersection number; 
in fact, they have opposite signs. 
On the other hand, 
each case of (i)--(iii) contains a single virtual crossing. 
It follows that the number of virtual crossings of $D$ 
in the cases (i)--(iii) is greater than or equal to $|n|$. 
\end{proof}

\begin{example}\label{ex510}
(i) 
Let $K_n$ $(n\geq1)$ be the virtual knot presented 
by the Gauss diagram and 
the virtual diagram  
as shown in Figure \ref{vc1}. 
We see that 
${\rm c}(K_n)=n+3$ and ${\rm vc}(K_n)=n+2$ can be 
detected by $I_{K_n}(t)$ 
but not by $W_{K_n}(t)$, $I\!I_{K_n}(t)$, 
and $I\!I\!I_{K_n}(t)$. 
In fact, we have 
\begin{eqnarray*}
W_{K_n}(t)&=&t^{n+1}-t^2-nt+n+1-t^{-1},\\
I_{K_n}(t)&=& -t^{n+2}+(n+1)t^{n+1}-t^n+(n+1)t-2\sum_{i=1}^{n}t^{i}, \\
I\!I_{K_n}(t)&=& n(t^{n+1}+t^{-n-1})-(n^2+n+2)(t+t^{-1})+2(n^2+2n+2)\\
& & -2\sum_{i=1}^{n}(t^{i}+t^{-i}), \mbox{ and}\\
I\!I\!I_{K_n}(t)&\equiv&n(t^2+t^{-2})-2(t+t^{-1})+4 
-\sum_{i=1}^{n}(t^i+t^{-i})
\quad \mbox{(mod~$\overline{W}_{K_n}(t)$)}, 
\end{eqnarray*}
where  
$\overline{W}_{K_n}(t)=(t^{n+1}+t^{-n-1})-(t^2+t^{-2})-(n+1)(t+t^{-1})+2n+2$.

\begin{figure}[htb]
\begin{center}
\includegraphics[bb = 0 0 326.48 103.93]{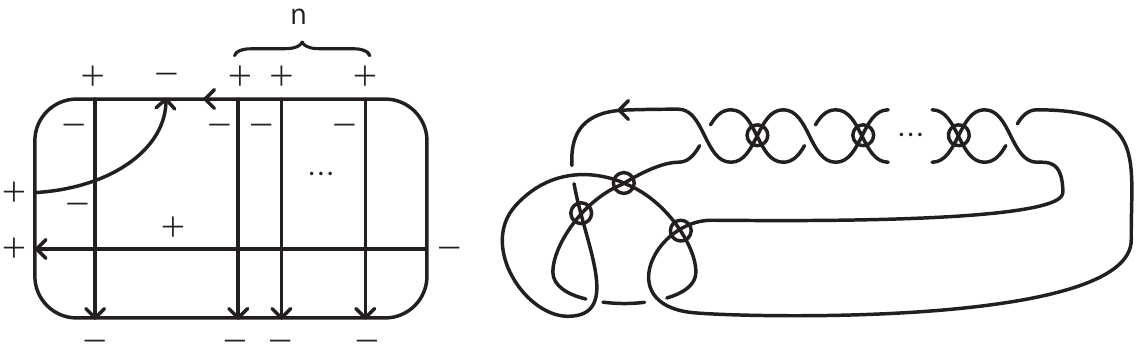}
\caption{}
\label{vc1}
\end{center}
\end{figure}

(ii) 
Let $K_n'$ $(n\geq1)$ be the virtual knot presented 
by the Gauss diagram and 
the virtual diagram 
as shown in Figure \ref{vc2}. 
We see that 
${\rm c}(K_n')=n+3$ and ${\rm vc}(K_n')=n+2$ can be detected by 
$I\!I_{K_n'}(t)$ 
but not by $W_{K_n'}(t)$, $I_{K_n'}(t)$, 
and $I\!I\!I_{K_n'}(t)$. 
In fact, we have 
\begin{eqnarray*}
W_{K_n'}(t)&=&t^{n+1}-nt+n-1-t^{-1}+t^{-2},\\
I_{K_n'}(t)&=& (n-1)t^{n+1}+(2n+1)t-n-1-(n-2)t^{-1}+(n-1)t^{-2}-2\sum_{i=1}^{n}t^{i}, \\
I\!I_{K_n'}(t)&=& (t^{n+2}+t^{-n-2})+(n-2)(t^{n+1}+t^{-n-1})\\
& & +(n-1)(t^{2}+t^{-2})-n(n+1)(t+t^{-1})+2(n^2+n+2)\\
& & -\sum_{i=1}^{n}(t^{i}+t^{-i})-\sum_{i=1}^{n}(t^{i-1}+t^{-i+1}), \mbox{ and}\\
I\!I\!I_{K_n'}(t)&\equiv& 
-n(t+t^{-1})+4n-\sum_{i=1}^{n}(t^{i-1}+t^{-i+1})
\quad \mbox{(mod~$\overline{W}_{K_n'}(t)$)}, 
\end{eqnarray*}
where 
$\overline{W}_{K_n'}(t)=(t^{n+1}+t^{-n-1})+(t^2+t^{-2})-(n-1)(t+t^{-1})+2n-2$.

\begin{figure}[htb]
\begin{center}
\includegraphics[bb = 0 0 326.48 103.93]{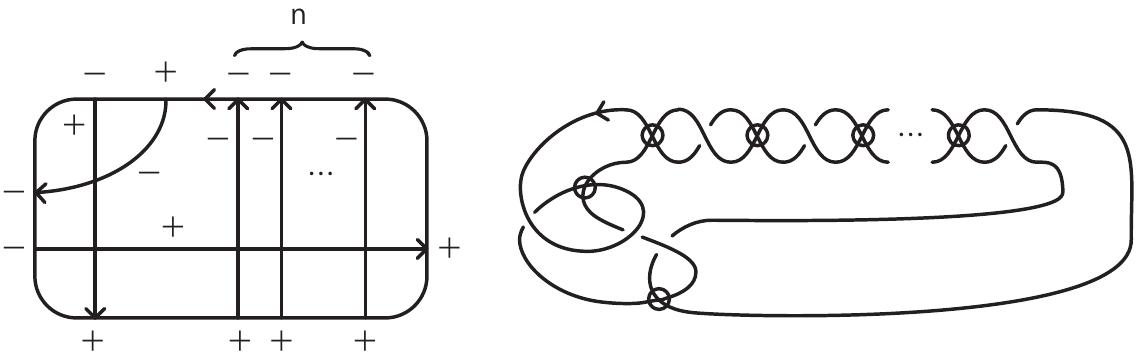}
\caption{}
\label{vc2}
\end{center}
\end{figure}

\end{example}



\appendix
\section{Table of $W_K(t)$, $I_K(t)$, and $I\!I_K(t)$.}\label{appA}
Table~\ref{table5} shows $W_K(t)$, $I_K(t)$, and $I\!I_K(t)$ 
of a virtual knot $K$ up to crossing number four 
according to Green's table~\cite{Gr} with a choice of orientations.
We use the following notations; 
$$\begin{array}{lcl}
\{n\}(a_0+a_1+\cdots+a_m) & =&
a_0t^{n}+a_1t^{n+1}+\cdots+a_mt^{n+m}
\ \ \mbox{ and }\\
{[}b_0+b_1+b_2+\cdots &=& 
b_0 + b_1(t+t^{-1}) + b_2(t^2+t^{-2})+\cdots,
\end{array}$$
where $m\geq 1$ and $a_0\neq 0$.

\begin{longtable}{|l|l|l|l|} \hline\label{table5}
  & $W_K(t)$                  & $I_K(t)$                       & $I\!I_K(t)$  \\ \hline \hline
 $2.1$  & $[2-1$                   &  $[2-1$                       &$[4-2$  \\ \hline
 $3.1$  & $\{-1\}(-1+1+1-1)$    &  $\{0\}(-1+2-1)$           & $[4-2$   \\ \hline
 $3.2$  & $[2-1$                     & $[2-1$                        &  $[4-2$    \\ \hline
 $3.3$  & $\{-1\}(-2+3+0-1)$    & $\{-1\}(-2+2+2-2)$       & $[10-4-1$   \\ \hline  
 $3.4$  & $\{0\}(1-2+1)$          &    $0$                          &  $[6-4+1$  \\ \hline
 $3.5$  & $[2+0-1$                 & $[0+2-2$                      & $[0+4-4$   \\ \hline  
 $3.6$  & $ 0$                        &   $ 0$                           &  $0$   \\ \hline  
 $3.7$  & $[2+0-1$                 &  $[4-2$                         & $[8-4$  \\ \hline  
 $4.1$  &  $[4-2$                    & $[4-2$                         &  $[20-12+2$   \\ \hline  
 $4.2$  & $0$                         &  $[4-2$                        &  $[-4+4-2$  \\ \hline  
 $4.3$  &  $[4-2$                    & $\{-1\}(-3+7-5+1)$        &  $[14-8+1$  \\ \hline  
 $4.4$  & $[2-1$                     & $\{0\}(-1+2-1)$             & $[10-6+1$ \\ \hline
 $4.5$  & $[2-1$                     &  $\{-1\}(-2+5-4+1)$       &  $[-2+2-1$   \\ \hline  
 $4.6$  & $0$                         & $\{-1\}(-1+1+1-1)$         &  $[2+0-1$  \\ \hline
 $4.7$  & $[4-2$                     & $[10-6+1$                     &  $[8-4$  \\ \hline  
 $4.8$  & $0$                         & $[-2+2-1$                     & $[8-4$ \\ \hline  
 $4.9$  & $[2-1$                     &  $[-2+1$                       &  $[2-2+1$   \\ \hline  
 $4.10$ & $\{-1\}(-1+1+1-1)$    & $\{-1\}(-1+0+3-2)$        &  $[0+2-2$  \\ \hline
 $4.11$ & $\{-1\}(-2+3+0-1)$    & $\{-1\}(-2+2+2-2)$        &  $[10-4-1$  \\ \hline  
 $4.12$ & $0$                         & $0$                              &  $0$ \\ \hline  
 $4.13$ & $0$                         &     $0$                          & $0$   \\ \hline  
 $4.14$ & $[0-1+1$                  & $[-2+1$                       &  $[6-3-1+1$  \\ \hline  
 $4.15$ & $\{-1\}(-2+3+0-1)$    & $\{-2\}(1-4+2+4-3)$       &  $[4+0-2$  \\ \hline  
 $4.16$ & $0$                          & $\{-1\}(1-3+3-1)$          & $0$ \\ \hline  
 $4.17$ & $\{-1\}(-1+1+1-1)$     & $\{0\}(-1+2-1)$             &  $[4-2$  \\ \hline
 $4.18$ & $[2-1$                      & $[2-1$                          &  $[4-2$  \\ \hline  
 $4.19$ & $\{-1\}(-1+1+1-1)$     & $\{0\}(-1+2-1)$              &   $[4-2$ \\ \hline
 $4.20$ & $\{0\}(1-2+1)$           & $0$                               &  $0$   \\ \hline 
 $4.21$ & $[-2+0+1$                 & $\{-2\}(-2+3-1+1-1)$      &  $[8-1-4+1$  \\ \hline  
 $4.22$ & $\{-1\}(-1+1+0+1-1)$  & $0$                              &  $[4-1-2+1$  \\ \hline  
 $4.23$ & $\{-1\}(-1+1+1-1)$     & $\{-1\}(-1+0+3-2)$         & $[0+2-2$ \\ \hline
 $4.24$ & $\{-2\}(1+0-1-1+0+1)$ & $\{-2\}(-1+1-2+3+1-2)$  & $[0+3-2-1$ \\ \hline  
 $4.25$ & $[4-2$                       &$\{-1\}(-3+7-5+1)$          & $[14-8+1$ \\ \hline
 $4.26$ & $\{-1\}(-1+0+2+0-1)$   & $\{-1\}(1-3+3-1)$          & $[10-5-1+1$ \\ \hline
 $4.27$ & $[2-1$                       & $\{-1\}(-2+3+0-1)$         & $[6-2-1$ \\ \hline
 $4.28$ & $\{-1\}(2-2-1+0+1)$     & $\{0\}(-3+4+1-2)$         & $[10-5+1-1$ \\ \hline 
 $4.29$ & $\{-1\}(-2+3+0-1)$      & $\{-1\}(-2+2+2-2)$       & $[10-4-1$ \\ \hline  
 $4.30$ & $[2-1$                       & $\{1\}(-1+2-1)$            & $[10-6+1$ \\ \hline
 $4.31$ & $0$                           & $0$                             & $0$ \\ \hline  
 $4.32$ & $\{-1\}(-1+1+1-1)$      & $[2-1$                        & $[4-2$ \\ \hline
 $4.33$ & $[2-1$                       & $[2-1$                        & $[10-6+1$ \\ \hline 
 $4.34$ & $\{0\}(1-2+1)$            & $0$                             & $[6-4+1$ \\ \hline
 $4.35$ & $\{-1\}(-1+1+1-1)$      & $[2-1$                        & $[4-2$ \\ \hline  
 $4.36$ & $[-2+0+1$                  & $[2+0-1$                     & $[4+0-2$ \\ \hline
 $4.37$ & $[4-1-1$                    & $[6-1-2$                     & $[12-2-4$ \\ \hline  
 $4.38$ & $\{0\}(1-2+1)$            & $\{1\}(2-4+2)$              & $[4-4+2$ \\ \hline
 $4.39$ & $\{-1\}(-1+1+1-1)$      & $\{1\}(-2+4-2)$            & $[4-3+2-1$ \\ \hline  
 $4.40$ & $[2-1$                       & $[2-1$                        & $[4-2$ \\ \hline
 $4.41$ & $0$                           & $0$                            & $0$ \\ \hline  
 $4.42$ & $\{0\}(1-1-1+1)$         & $\{1\}(-1+2-1)$            & $[4-2$ \\ \hline
 $4.43$ & $[4-2$                       & $[8-4$                        & $[16-8$ \\ \hline 
 $4.44$ & $[2-1$                       & $\{0\}(1-2+1)$             & $[2-2+1$ \\ \hline
 $4.45$ & $\{-1\}(-2+2+1+0-1)$   & $\{0\}(-2+3+0-1)$        & $[12-6$ \\ \hline
 $4.46$ & $0$                           & $\{-1\}(-1+1+1-1)$       & $[2+0-1$ \\ \hline
 $4.47$ & $\{-1\}(1+0-2+0+1)$    & $\{-1\}(1-4+4+0-1)$      & $[8-4$ \\ \hline  
 $4.48$ & $[4-1-1$                    & $[2+1-2$                     & $[14-3-5+1$ \\ \hline
 $4.49$ & $\{0\}(1-2+1)$            & $0$                             & $[6-4+1$ \\ \hline  
 $4.50$ & $\{-1\}(-1+1+1-1)$      & $\{-1\}(-1+0+3-2)$       & $[0+2-2$ \\ \hline
 $4.51$ & $0$                           & $0$                             & $0$ \\ \hline  
 $4.52$ & $[2-1$                       & $[-2+1$                       & $[2-2+1$ \\ \hline
 $4.53$ & $[4-2$                       & $[6-4+1$                     & $[12-8+2$ \\ \hline
 $4.54$ & $[2-1$                       & $\{-1\}(-2+3+0-1)$       & $[6-2-1$ \\ \hline
 $4.55$ & $0$                           & $[-4+2$                       & $[4-4+2$ \\ \hline  
 $4.56$ & $0$                           & $\{0\}(2-4+2)$              & $[-8+4$ \\ \hline
 $4.57$ & $\{-1\}(-1+1+1-1)$      & $\{0\}(-1+2-1)$            & $[4-2$ \\ \hline  
 $4.58$ & $0$                           & $\{-1\}(-1+1+1-1)$       & $[8-4$ \\ \hline
 $4.59$ & $0$                           & $\{0\}(-1+1+1-1)$         & $[8-4$ \\ \hline  
 $4.60$ & $[2-1$                       & $[2-1$                        & $[4-2$ \\ \hline
 $4.61$ & $[2-1$                       & $[2-1$                        & $[4-2$ \\ \hline  
 $4.62$ & $\{-2\}(-1-1+3+0+0-1)$& $\{-1\}(-3+4+0+0-1)$   & $[12-4-2$ \\ \hline
 $4.63$ & $\{-1\}(-2+3+0-1)$      & $\{-2\}(1-4+4+0-1)$     & $[8-4$ \\ \hline  
 $4.64$ & $[0-1+1$                    & $[2-1$                       & $[4-2$ \\ \hline
 $4.65$ & $[2+0-1$                    & $[2+0-1$                    & $[4+0-2$ \\ \hline
 $4.66$ & $\{-2\}(1+0-1-1+0+1)$ & $\{-2\}(1-3+2+1-1)$     & $[8-3-2+1$ \\ \hline
 $4.67$ & $\{-1\}(1-1-1+1)$        & $[2-1$                       & $[4-2$ \\ \hline  
 $4.68$ & $0$                            & $0$                           & $0$ \\ \hline
 $4.69$ & $[2-1$                        & $[2-1$                       & $[4-2$ \\ \hline  
 $4.70$ & $\{-1\}(-1+1+1-1)$       & $\{0\}(-1+2-1)$           & $[4-2$ \\ \hline
 $4.71$ & $0$                            & $\{-2\}(-1+2-1-1+2-1)$& $[8-4$ \\ \hline  
 $4.72$ &  $0$                           & $\{-1\}(-1+1+1-1)$      & $[8-4$ \\ \hline
 $4.73$ & $[4-2$                        & $[8-4$                       & $[16-8$ \\ \hline  
 $4.74$ & $[2-1$                        & $\{0\}(1-2+1)$            & $[2-2+1$ \\ \hline
 $4.75$ & $0$                            & $[2+0-1$                    & $[4+0-2$ \\ \hline  
 $4.76$ & $0$                            & $[4-2$                       & $[-4+4-2$ \\ \hline
 $4.77$ & $0$                            & $\{0\}(-2+4-2)$           & $[8-4$ \\ \hline  
 $4.78$ & $\{-2\}(-1-1+3+0+0-1)$& $\{-2\}(-2+1+0+2+2-3)$& $[4+2-2-2$ \\ \hline  
 $4.79$ & $\{0\}(1-1-1+1)$          & $\{1\}(1-2+1)$             & $[4+0-4+2$ \\ \hline
 $4.80$ & $\{-1\}(-3+4+0+0-1)$    & $\{-1\}(-3+2+2+2-3)$   & $[20-10+2-2$ \\ \hline  
 $4.81$ & $\{0\}(2-3+0+1)$            & $\{1\}(1-2+1)$            & $[16-8-2+2$  \\ \hline  
 $4.82$ & $[4-1-1$                     & $[6-1-2$                    & $[12-2-4$ \\ \hline
 $4.83$ & $\{-2\}(-1+0+2-1+1-1)$ & $\{-2\}(-1+1+0-1+3-2)$& $[4-1+0-1$ \\ \hline  
 $4.84$ & $[0+1-1$                      & $[2-1$                       & $[4-2$ \\ \hline
 $4.85$ & $[2+0-1$                     & $\{-2\}(-1+1+0+0+1-1)$ & $[8-1-4+1$ \\ \hline
 $4.86$ & $[2+0-1$                     & $[4-2$                        & $[8-4$ \\ \hline  
 $4.87$ & $\{-2\}(-2+0+4-1+0-1)$ & $\{-2\}(-4+2+4-1+2-3)$ & $[12+0-4-2$ \\ \hline
 $4.88$ & $\{1\}(1-2+1)$             & $\{1\}(-1+2-1)$             & $[4-2$  \\ \hline  
 $4.89$ & $[4+0-2$                     & $[4+2-4$                     & $[16+2-12+2$ \\ \hline
 $4.90$ & $0$                            & $\{1\}(-2+4-2)$            & $[8-4$ \\ \hline  
 $4.91$ & $[4-1+0-1$               & $[4-1+2-3$                  & $[8-2+4-6$ \\ \hline
 $4.92$ & $[2+0+0-1$               & $[-2+2+2-3$                & $[-4+4+4-6$ \\ \hline
 $4.93$ & $\{-1\}(-1+2+0-2+1)$  & $\{-1\}(-1+2+0-2+1)$   & $[8-2-4+2$ \\ \hline
 $4.94$ & $[2-1$                      & $[2-1$                        & $[4-2$  \\ \hline    
 $4.95$ & $[2+0+0-1$               & $[2+0+0-1$                  & $[4+0+0-2$ \\ \hline
 $4.96$ & $[2+0-1$                  & $[2+0-1$                     & $[4+0-2$ \\ \hline  
 $4.97$ & $\{-2\}(1-1+0+0-1+1)$ & $0$                             & $[4-1-2+1$ \\ \hline
 $4.98$ & $0$                           & $[0-1+2-1$                 & $[8-4$ \\ \hline
 $4.99$ & $0$                           & $[8-4$                         & $[16-8$  \\ \hline  
 $4.100$ & $[4-2$                     & $[8-4$                         & $[16-8$ \\ \hline
 $4.101$ & $[2+0+0-1$               & $[2+0+0-1$                   & $[4+0+0-2$ \\ \hline  
 $4.102$ & $[0-1+0+1$               & $[4-1-2+1$                   & $[8-2-4+2$ \\ \hline
 $4.103$ & $\{-2\}(-2+1+2+0+0-1)$& $\{-2\}(-2+1+2+0+0-1)$& $[8+0-4$ \\ \hline
 $4.104$ & $[-2+0+0+1$              & $[6-2-2+1$                  & $[12-4-4+2$\\ \hline
 $4.105$ & $0$                           & $[-8+4$                      & $[-16+8$ \\ \hline
 $4.106$ & $[2+0-1$                    & $[0+2-2$                    & $[0+4-4$ \\ \hline
 $4.107$ & $0$                           & $[-4+2$                      & $[0+2-4+2$\\ \hline
 $4.108$ & $0$                           & $0$                            & $0$ \\ \hline
\caption{}
\end{longtable}

\section{Table of $\overline{W}_K(t)$, $f_{00}(D;t)$, 
$f_{11}(D;t)$, and $I\!I\!I_K(t)$.}\label{appB}
Table~\ref{table6} shows $\overline{W}_K(t)$, $f_{00}(D;t)$, 
$f_{11}(D;t)$, and $I\!I\!I_K(t)$
of a virtual knot $K$ up to crossing number four 
with a choice of orientations.
We remark that these polynomials are all reciprocal.

\begin{longtable}{|l|l|l|l|l|} \hline\label{table6}
    & $\overline{W}_K(t)$   & $f_{00}(D;t)$     & $f_{11}(D;t)$          & $I\!I\!I_K(t)$  \\ \hline \hline
 $2.1$  & $[4-2$                  &  $[-2+1$         &$[-2+1$                 & $[-2+1$ \\ \hline
 $3.1$  & $[2+0-1$                 &  $[2-2+1$      & $0$                       &  $[4-2$       \\ \hline
 $3.2$  & $[4-2$                   &  $[2-1$           & $[-2+1$                 &  $[2-1$       \\ \hline
 $3.3$  & $[6-2-1$                &  $[-4+1+1$      & $[-4+1+1$             &  $[2-1$       \\ \hline
 $3.4$  & $[2-2+1$                 &  $[2-1$          & $[2-1$                 &  $[2-1$       \\ \hline
 $3.5$  & $[4+0-2$                &  $[-6+2+1$       & $[-6+2+1$            &  $[-6+2+1$    \\ \hline
 $3.6$  & $0$                        &  $0$              & $0$                       &  $0$    \\ \hline
 $3.7$  & $[4+0-2$                &  $[2-2+1$       &  $[2-2+1$              &   $[2-2+1$  \\ \hline
 $4.1$  & $[8-4$                   &  $[-8+2+2$      &$[-4+2$                 &  $[0-2+2$ \\\hline
 $4.2$  & $0$                       &  $[0+2-2$       & $[-4+2$                &   $[0+2-2$ \\\hline
 $4.3$  &  $[8-4$                  & $[-8+4$         &  $[-10+4+1$           & $0$  \\ \hline  
 $4.4$  & $[4-2$                   & $[2-1$          & $[0-1+1$                &  $[2-1$   \\ \hline
 $4.5$  & $[4-2$                   &  $[-4+3-1$     &  $[-6+3$                &$[0+1-1$   \\ \hline    
 $4.6$  & $0$                       & $[2+0-1$        &  $0$                      & $[2+0-1$ \\ \hline
 $4.7$  & $[8-4$                   & $[-12+6$        &  $[-12+6$              &  $[-4+2$ \\ \hline  
 $4.8$  & $0$                       & $[4-2$           & $[4-2$                  & $[4-2$  \\ \hline  
 $4.9$  & $[4-2$                   &  $[-10+5$      &  $[-4+1+1$            & $[-2+1$  \\ \hline  
 $4.10$ &$[2+0-1$                & $0$               &  $[-4+2$               & $0$  \\ \hline
 $4.11$ & $[6-2-1$               & $[0-1+1$       &  $[-2+1$                & $[6-3$ \\ \hline  
 $4.12$ & $0$                      & $0$              &  $0$                      &  $0$ \\ \hline 
 $4.13$ & $0$                      &     $[-4+2$    & $[4-2$                   &  $[-4+2$   \\ \hline  
 $4.14$ & $[0-2+2$               & $[2-1$         &  $[4-2-1+1$            & $[2-1$ \\ \hline  
 $4.15$ & $[6-2-1$               & $[-8+3+1$    & $[-12+5+1$             & $[-2+1$ \\ \hline  
 $4.16$ & $0$                             &  $0$                                     & $0$                                 & $0$ \\ \hline  
 $4.17$ & $[2+0-1$            & $0$                                     &  $[4-2$                 & $0$ \\ \hline
 $4.18$ & $[4-2$            & $[-2+1$                         &  $[-2+1$                      & $[-2+1$    \\ \hline  
 $4.19$ & $[2+0-1$             &  $0$                                   &   $[4-2$                &  $0$ \\ \hline
 $4.20$ & $[2-2+1$& $[-2+1$                        & $[2-1$                     & $[-2+1$    \\ \hline 
 $4.21$ & $[-4+0+2$            &  $[-2+0+1$                  &  $[2-1-1+1$ &  $[-2+0+1$ \\ \hline  
 $4.22$ & $[2-1+1-1$ & $[2+0-2+1$ &  $[2-1$& $[4-1-1$\\ \hline  
 $4.23$ & $[2+0-1$        & $[-4+2$                         &  $0$                          & $[-4+2$   \\ \hline
 $4.24$ & $[-2-1+1+1$ & $[-2+1-1+1$ & $[-2+0+1$ & $[0+2-2$\\\hline  
 $4.25$ & $[8-4$             & $[-10+4+1$        & $[-8+4$ & $[-2+0+1$  \\ \hline
 $4.26$ & $[0+1+0-1$ & $[6-3-1+1$ & $[4-2$ & $[6-2-1$\\\hline
 $4.27$ & $[4-2$             &$[-2+1$                   & $[0+1-1$ &$[-2+1$\\ \hline
 $4.28$ & $[-4+1+0+1$   &$[2-2+0+1$  & $[0-1+1$  &$[6-3$  \\ \hline 
 $4.29$ & $[6-2-1$   & $[-6+1+2$      & $[-8+3+1$ & $[6-3$    \\ \hline  
 $4.30$ & $[4-2$            & $[0-1+1$      & $[2-1$      & $[0-1+1$ \\ \hline
 $4.31$ & $0$                        & $[4-2$                         & $[-4+2$        & $[4-2$ \\ \hline  
 $4.32$ & $[2+0-1$     & $[4-2$                    &  $0$    & $[4-2$  \\ \hline
 $4.33$ & $[4-2$           & $[-2+1$                   & $[4-3+1$        & $[-2+1$  \\ \hline 
 $4.34$ & $[2-2+1$    &  $[4-3+1$  & $[2-1$   &  $[2-1$  \\ \hline
 $4.35$ & $[2+0-1$        &  $0$                 & $[4-2$    &  $0$   \\ \hline  
 $4.36$ & $[-4+0+2$                & $[-2+0+1$                  & $[-2+0+1$         & $[-2+0+1$  \\ \hline
 $4.37$ & $[8-2-2$   & $[-10+3+2$  & $[-10+3+2$& $[-2+1$ \\ \hline  
 $4.38$ &  $[2-2+1$                    & $[-2+1$                     & $[2-1$    & $[-2+1$     \\ \hline
 $4.39$ & $[2-1+1-1$   & $[0+0-1+1$                & $[0-1+1$   &$[2-1$  \\ \hline  
 $4.40$ & $[4-2$                 & $[2-1$                    & $[2-1$   & $[2-1$   \\ \hline
 $4.41$ & $0$                              & $0$                                 & $0$                 & $0$  \\ \hline  
 $4.42$ & $[2-1-1+1$              &  $0$                 & $[4-2$            &  $0$     \\ \hline
 $4.43$ & $[8-4$             & $[-8+4$                 & $[-8+4$ &  $0$    \\ \hline 
 $4.44$ & $[4-2$               &  $[-4+1+1$                     & $[-2+1$  &  $[0-1+1$ \\ \hline
 $4.45$ & $[4-1+0-1$    & $[2-2+0+1$    & $[2-2+0+1$& $[6-1-2$\\\hline
 $4.46$ & $0$                            & $0$             &  $[2+0-1$    & $0$     \\ \hline
 $4.47$ & $[0-1+0+1$               & $[4-2$            & $[4-2$    & $[4-2$\\ \hline  
 $4.48$ & $[8-2-2$ & $[-6+2+0+1$  & $[-12+3+3$ & $[-6+2+0+1$\\\hline
 $4.49$ & $[2-2+1$   &  $[0+1-1$    & $[2-1$ &  $[2-1$ \\ \hline  
 $4.50$ & $[2+0-1$        &  $0$         & $[-4+2$     &  $0$ \\ \hline
 $4.51$ & $0$                              & $[-4+2$                               & $[4-2$      & $[-4+2$       \\ \hline  
 $4.52$ & $[4-2$                 & $[-2+1$                     &  $[4-3+1$      & $[-2+1$ \\ \hline
 $4.53$ & $[8-4$           & $[-10+4+1$   & $[-10+4+1$   & $[-2+0+1$ \\ \hline
 $4.54$ & $[4-2$             &  $[0+1-1$               &$[-2+1$    &  $[0+1-1$ \\ \hline
 $4.55$ & $0$                            &$[2-2+1$     & $[2-2+1$   &$[2-2+1$ \\ \hline  
 $4.56$ & $0$                           & $[-4+2$                       & $[-4+2$      & $[-4+2$    \\ \hline
 $4.57$ & $[2+0-1$        & $[-2+0+1$                    & $[2-2+1$ &  $0$  \\ \hline  
 $4.58$ & $0$                           & $[4-2$              & $[4-2$       & $[4-2$ \\ \hline
 $4.59$ & $0$                           & $[4-2$                 & $[4-2$ & $[4-2$  \\ \hline  
 $4.60$ & $[4-2$              & $[2-1$                   & $[2-1$   & $[2-1$   \\ \hline
 $4.61$ & $[4-2$                & $[-6+3$                    & $[-6+3$    & $[-2+1$   \\ \hline  
 $4.62$ & $[6-1-1-1$  & $[-2+0+0+1$     & $[2-2+0+1$& $[4-1-1$   \\ \hline
 $4.63$ & $[6-2-1$        & $[0-1+1$       &  $[-4+1+1$ & $[6-3$ \\ \hline  
 $4.64$ &  $[0-2+2$     & $[2-1$                    & $[2-1$  & $[2-1$   \\ \hline
 $4.65$ & $[4+0-2$            & $[-2+0+1$                  &$[-2+0+1$     & $[-2+0+1$         \\ \hline
 $4.66$ & $[-2-1+1+1$ & $[4-1-2+1$ & $[4-2$ & $[6+0-3$ \\ \hline
 $4.67$ & $[-2+0+1$           & $[4-2$                    & $0$ & $[4-2$   \\ \hline   
 $4.68$ & $0$                            & $0$                                   & $0$         & $0$ \\ \hline
 $4.69$ & $[4-2$              & $[-6+3$                     & $[-6+3$& $[-2+1$\\ \hline  
 $4.70$ & $[2+0-1$        & $[2-2+1$                        & $[-2+0+1$   & $[4-2$ \\ \hline
 $4.71$ & $0$                           & $[4-2$ & $[4-2$      & $[4-2$ \\ \hline  
 $4.72$ &  $0$                          & $[4-2$                    & $[4-2$     & $[4-2$    \\ \hline
 $4.73$ & $[8-4$          & $[-8+4$                    & $[-8+4$ & $0$   \\ \hline  
 $4.74$ & $[4-2$             &  $[-4+1+1$                 & $[-2+1$ &  $[0-1+1$\\ \hline
 $4.75$ & $0$                           & $[2+0-1$                 & $[2+0-1$  & $[2+0-1$  \\ \hline  
 $4.76$ & $0$                          & $[-2+2-1$   &  $[-2+2-1$   & $[-2+2-1$    \\ \hline
 $4.77$ & $0$                           & $[4-2$                      & $[4-2$     & $[4-2$  \\ \hline  
 $4.78$ &$[6-1-1-1$&$[-12+4+1+1$&$[-8+2+1+1$&$[-6+3$\\ \hline  
 $4.79$ & $[2-1-1+1$   &  $[-2+2-1$ & $[2+0-1$  &  $[-2+2-1$  \\ \hline
 $4.80$ & $[8-3+0-1$     & $[-6+1+1+1$   & $[-6+1+1+1$ & $[2-2+1$ \\ \hline  
 $4.81$ & $[4-3+0+1$        &  $[4-1-1$ &  $[4-1-1$   &  $[4-1-1$ \\ \hline  
 $4.82$ &$[8-2-2$& $[-8+3+1$& $[-12+3+3$  & $[-8+3+1$ \\ \hline
 $4.83$ & $[4-1+0-1$        & $[-2+0+0+1$    & $[-2+1$   & $[2-1$ \\ \hline  
 $4.84$ & $[0+2-2$    & $[4-1-1$            &$[0-1+1$  & $[4-1-1$\\ \hline
 $4.85$ & $[4+0-2$                 & $[0-1+0+1$         &  $0$       & $[0-1+0+1$\\ \hline
 $4.86$ &$[4+0-2$              &$[6-2-1$      & $[2-2+1$  &$[6-2-1$ \\ \hline  
 $4.87$ & $[8-1-2-1$  & $[-10+2+2+1$ & $[-10+2+2+1$ & $[-2+1$\\ \hline
 $4.88$ & $[0+1-2+1$           & $[2-1$         & $[2-1$  & $[2-1$   \\ \hline  
 $4.89$ & $[8+0-4$     &    $[-8+1+2+1$    &  $[-8+1+2+1$& $[0+1-2+1$ \\ \hline
 $4.90$ & $0$                               & $[4-2$                    & $[4-2$    & $[4-2$    \\ \hline  
 $4.91$ & $[8-2+0-2$    & $[-12+3+2+1$ &  $[-12+3+2+1$ & $[-12+3+2+1$  \\ \hline
 $4.92$ & $[4+0+0-2$             & $[-10+2+2+1$ &  $[-10+2+2+1$ & $[-10+2+2+1$ \\ \hline
 $4.93$ & $[4-1-2+1$    &  $0$                 &  $0$  &  $0$    \\ \hline
 $4.94$ & $[4-2$              &$[-6+3$                         & $[-6+3$   &$[-2+1$        \\ \hline   
 $4.95$ & $[4+0+0-2$            &$[-2+0+0+1$                       & $[-2+0+0+1$  &$[-2+0+0+1$   \\ \hline
 $4.96$ &$[4+0-2$             &  $0$                   & $[-4+2$   &  $0$  \\ \hline  
 $4.97$ & $[0-1+0+1$       &  $[2+0-1$                &$[2-1-1+1$   &  $[2+0-1$ \\ \hline
 $4.98$ & $0$                             & $[4-2$ & $[4-2$   & $[4-2$\\ \hline
 $4.99$ & $0$                             & $[8-4$                     & $[8-4$  & $[8-4$  \\ \hline  
 $4.100$ & $[8-4$           & $[-8+4$                   & $[-8+4$  & $0$  \\ \hline
 $4.101$ & $[4+0+0-2$             &$[-2+0+0+1$                    & $[-2+0+0+1$  &$[-2+0+0+1$   \\ \hline  
 $4.102$  & $[0-2+0+2$    & $[4-1-2+1$ & $[4-1-2+1$ & $[4-1-2+1$ \\ \hline
 $4.103$  & $[4+1-2-1$    & $[0-1+0+1$   & $[0-1+0+1$& $[4+0-2$ \\ \hline
 $4.104$  & $[-4+0+0+2$    & $[6-2-2+1$ & $[6-2-2+1$  & $[6-2-2+1$ \\ \hline
 $4.105$ & $0$                           & $[-8+4$                    & $[-8+4$ & $[-8+4$       \\ \hline
 $4.106$ & $[4+0-2$    & $[-6+2+1$   &  $[-2+2-1$  & $[-6+2+1$  \\ \hline
 $4.107$ & $0$                     & $[0+2-2$          & $[0+0-2+2$   & $[0+2-2$\\ \hline
 $4.108$ & $0$                       & $0$                                & $0$          & $0$ \\ \hline
\caption{}
\end{longtable}

\end{document}